\documentclass[12pt, a4paper]{amsart}
\usepackage{amsmath,amssymb,amscd,amsfonts}
\newtheorem{theorem}{Theorem}[section]
\newtheorem{rem}[theorem]{Remark}
\newtheorem{defn}[theorem]{Definition}
\newtheorem{lemma}[theorem]{Lemma}
\newtheorem{claim}[theorem]{Claim}
\newtheorem{proposition}[theorem]{Proposition}
\newtheorem{conjecture}[theorem]{Conjecture}
\newtheorem{corollary}[theorem]{Corollary}

\newcommand\Q{{\mathbb{Q}}}
\newcommand\C{{\mathbb{C}}}

\def\Psef{\mathop{\rm Psef}\nolimits}
\def\Vol{\mathop{\rm Vol}\nolimits}

\def\Pic{\mathop{\rm Pic}\nolimits}

\def\ddbar{\partial\overline\partial}
\def\cO{{\mathcal O}}
\def\cE{{\mathcal E}}
\def\cP{{\mathcal P}}
\let\ol=\overline
\let\wt=\widetilde

\def\bQ{{\mathbb Q}}
\def\bC{{\mathbb C}}
\def\bR{{\mathbb R}}

\begin{document}
\title[Extension theorems]{Extension theorems, Non-vanishing  and the existence of good minimal models}

\author{Jean-Pierre Demailly, Christopher D. Hacon \\ and Mihai P\u aun}
\address{Universit\'e de Grenoble I \\
D\'epartement de Math\'ematiques \\
Institut Fourier\\
38402 Saint-Martin d'H\`eres, France}
\email{demailly@fourier.ujf-grenoble.fr}
\address{Department of Mathematics \\
University of Utah\\
155 South 1400 East\\
JWB 233\\
Salt Lake City, UT 84112, USA}
\email{hacon@math.utah.edu}

\address{Institut Elie Cartan \\
Universit\'e Henri Poincar\'e\\
B. P. 70239, F-54506 Vandoeuvre-l\`es-Nancy Cedex, France}
\email{Mihai.Paun@iecn.u-nancy.fr}
\date{\today}
\thanks{The second author was partially supported by NSF research grant no: 0757897. During an important part of the preparation of this article, the third named author was visiting KIAS (Seoul); he wishes to express his gratitude for the support and excellent working conditions provided by this institute. We would like to thank F.~Ambro, B.~Berndtsson and C.~Xu for interesting conversations about this article.}

\begin{abstract} 
We prove an extension theorem for effective plt pairs $(X,S+B)$ of non-negative Kodaira dimension $\kappa (K_X+S+B)\geq 0$. The main new ingredient is a refinement of the Ohsawa-Takegoshi $L^2$ extension theorem involving singular hermitian metrics. 
\end{abstract}
\maketitle

\section{Introduction}
Let $X$ be a complex projective variety with mild singularities. The aim of the minimal model program is to produce a birational map $X\dasharrow X'$ such that: \begin{enumerate}\item If $K_X$ is pseudo-effective, then $X'$ is a good minimal model so that $K_{X'}$ is
semiample; i.e. there is a morphism $X'\to Z$ and $K_{X'}$ is the pull-back of an ample $\mathbb Q$-divisor on $Z$.
\item If $K_X$ is not  pseudo-effective, then there exists a Mori-Fano fiber space $X'\to Z$, in particular $-K_{X'}$ is relatively ample.
\item The birational map $X\dasharrow X'$ is to be constructed out of a finite sequence of well understood ``elementary'' birational maps known as flips and divisorial contractions.\end{enumerate}

The existence of flips was recently established in \cite{BCHM10} where it is also proved that if $K_X$ is big then $X$ has a good minimal model and if $K_X$ is not pseudo-effective then there is a Mori-Fano fiber space.
The focus of the minimal model program has therefore shifted to varieties (or more generally log pairs) such that $K_X$ is pseudo-effective but not big.
\begin{conjecture}[Good Minimal Models]\label{c-gmm} Let $(X,\Delta )$ be an $n$-dimensional klt pair. If $K_X+\Delta $ is pseudo-effective then $(X,\Delta)$ has a good minimal model.
\end{conjecture}

Note that in particular the existence of good minimal models for log pairs would imply the following conjecture (which is known in dimension $\leq 3$ cf. \cite{KMM94}, \cite{Kollaretal92}):
\begin{conjecture}[Non-Vanishing]\label{c-abb} Let $(X,\Delta )$ be an $n$-dimensional klt pair. If $K_X+\Delta $ is pseudo-effective then $\kappa(K_X+\Delta )\geq 0$.
\end{conjecture}
It is expected that \eqref{c-gmm} and \eqref{c-abb} also hold in the more general context of log canonical (or even semi-log canonical) pairs $(X,\Delta )$.
Moreover, it is expected that the Non-Vanishing Conjecture implies existence of good minimal models. The general strategy for proving that \eqref{c-abb} implies \eqref{c-gmm} is explained in \cite{Fujino00}.
One of the key steps is to extend pluri-log canonical divisors from a divisor to the ambient variety. The key ingredient is the following.
\begin{conjecture}[DLT Extension]\label{c-dlt}
Let $(X,S+B)$ be an $n$-dimensional dlt pair such that $\lfloor S+B \rfloor =S$, $K_X+S+B$ is nef and $K_X+S+B\sim _{\mathbb Q}D\geq 0$ where $S\subset {\rm Supp}(D)$.
Then $$H^0(X,\mathcal O _X(m(K_X+S+B)))\to H^0(S,  \mathcal O _S(m(K_X+S+B)))  $$
is surjective for all $m>0$ sufficiently divisible.\end{conjecture}
We then have the following easy consequence (cf. \eqref{t-good}):
\begin{theorem}\label{t-good1} Assume \eqref{c-dlt}$_n$ holds and that \eqref{c-abb}$_{n}$ holds for all semi-log canonical pairs.
Then \eqref{c-gmm}$_n$ holds (i.e. \eqref{c-gmm} holds in dimension $n$).\end{theorem}
The main purpose of this article is to prove that Conjecture \ref{c-dlt} holds under the additional assumption that $(X,S+B)$ is plt, see Theorem \ref{t-ext} below.

\begin{rem} \eqref{c-abb} is known to hold in dimension $\leq 3$ cf. \cite{Kawamata92}, \cite{Miyaoka88}, \cite{KMM94}, \cite{Fujino00} and when $K_X+\Delta$ is nef and $\kappa _\sigma (K_X+\Delta )=0$ cf. \cite{Nakayama04}. See also \cite{Ambro04} and \cite{Fukuda02} for related results. A proof of the case when $X$ is smooth and $\Delta =0$ has been announced in \cite{Siu09} (this is expected to imply the general case cf. \eqref{t-abbw}, ).

The existence of a good minimal models is known for canonical pairs $(X,0)$ where $K_X$ is nef and $\kappa (K_X)=\nu (K_X)$ cf. \cite{Kawamata85a}, when $\kappa (K_X)=\dim (X)$ by \cite{BCHM10} and when the general fiber of the Iitaka fibration has a good minimal model by \cite{Lai10}.

Birkar has shown that \eqref{c-abb} implies the existence of minimal models (resp. Mori-Fano fiber spaces) and the existence of the corresponding sequence of flips and divisorial contractions cf. \cite[1.4]{Birkar09}. The existence of minimal models for klt $4$-folds is proven in
\cite{Shokurov09}.
\end{rem}
We also recall the following important consequence of \eqref{c-gmm} (cf. \cite{Birkar09}).
\begin{corollary}\label{c-termination} Assume \eqref{c-gmm}$_n$. Let $(X,\Delta )$ be an $n$-dimensional klt pair and $A$ an ample divisor such that $K_X+\Delta +A$ is nef. Then any $K_X+\Delta$-minimal model program with scaling terminates.\end{corollary}
\begin{proof}If $K_X+\Delta$ is not pseudo-effective, then the claim follows by \cite{BCHM10}.

If $K_X+\Delta$ is pseudo-effective, then by \eqref{t-good1}, we may assume that $(X,\Delta )$ has a good minimal model.
The result now follows from \cite{Lai10}.
\end{proof} 
We now turn to the description of the main result of this paper (cf. \eqref{t-ext}) which we believe is of independent interest.

Let $X$ be a smooth variety, and let $S+B$ be a $\mathbb Q$-divisor with simple normal crossings, such that $S=\lfloor S+B \rfloor$,
$$K_X+S+B\in \Psef(X)\qquad {\rm and} \quad S\not\subset N_\sigma(K_X+S+B).$$
We consider $\pi: \widetilde X\to X$ a log-resolution of $(X, S+B)$, so that we have
$$K_{\wt X}+ \wt S+ \wt B= \pi^\star(K_X+S+B)+ \wt E$$
where $\wt S$ is the proper transform of $S$. Moreover $\wt B$ and $ \wt E$ are effective $\Q$-divisors, the components of $\wt B$ are disjoint and $\wt E$ is $\pi$-exceptional.

Following \cite{HM10} and \cite{Paun08},
if we consider the \emph{extension obstruction} divisor
$$\Xi:= N_\sigma(\Vert K_{\wt X}+ \wt S+ \wt B\Vert _{\wt S})\wedge \wt B|_{\wt S},$$
then we have the following result.
\medskip

\begin{theorem}[Extension Theorem]\label{t-ext} Let $X$ be a smooth variety, $S+B$ a $\mathbb Q$-divisor with simple normal crossings such that
\begin{enumerate}
\item $(X,S+B)$ is plt (i.e. $S$ is a prime divisor with ${\rm mult}_S(S+B)=1$ and $\lfloor B \rfloor =0$),
\item there exists  an effective $\mathbb Q$-divisor $D\sim _{\mathbb Q}K_X+S+B$ such that $S\subset {\rm Supp }(D)\subset {\rm Supp }(S+B)$, and
\item $S$ is not contained in the support of $N_\sigma (K_X+S+B)$ (i.e., for any ample divisor $A$ and any rational number $\epsilon >0$, there is an effective $\mathbb Q$-divisor $D\sim _{\mathbb Q}K_X+S+B+\epsilon A$ whose support does not contain $S$).
\end{enumerate}
Let $m$ be an integer, such that $m(K_X+S+B)$ is Cartier, and let 
$u$ be a section of $m(K_X+S+B)|_S$, such that 
$$Z_{\pi^\star(u)}+ m\wt E|_{\wt S}\geq m\Xi,$$
where we denote by $\displaystyle Z_{\pi^\star(u)}$ the zero divisor of the section 
$\pi^\star(u)$. Then $u$ extends to $X$.
\end{theorem}
The above theorem is a strong generalization of similar results available in the literature (see for example \cite{Siu98}, \cite{Siu00}, \cite{Taka06}, \cite{Taka07}, \cite{HM07}, \cite{Paun07},  \cite{Clod07}, \cite{EP07}, \cite{dFH09}, \cite{Var08}, \cite{Paun08},  
\cite{Tsuji}, \cite{HM10}, \cite{BP10}).
The main and important difference is that we do not require any \emph{strict positivity} from $B$. The positivity of $B$ (typically one requires that $B$ contain an ample $\mathbb Q$-divisor) is of great importance in the algebraic approach
as it allows us to make use of the Kawamata-Viehweg Vanishing Theorem.
It is for this reason that \emph{so far} we are unable to give an algebraic proof of \eqref{t-ext}. 
In order to understand the connections between \eqref{t-ext} and the results quoted above, we mention here
that under the hypothesis of Theorem \ref{t-ext}, one knows that the section $u^{\otimes k}\otimes s_A$ extends to $X$, for each $k$ and each section $s_A$ of a sufficiently ample line bundle $A$. Our contribution is to show that a family 
of extensions can be constructed with a very precise estimate of their norm, as
$k\to\infty$. In order to obtain this special extensions we first prove a 
generalization of the version of the Ohsawa-Takegoshi Theorem (cf. \cite{OT87},
\cite{Oh03}, \cite{Oh04}) established in \cite{Man93}, \cite{Var08}, \cite{MV07} in which the existence of the divisor $D$,
together with the hypothesis (3) replace the strict positivity of $B$. By a limit process justified by the estimates we have just mentioned together with the classical results in \cite{Lelong69}, we obtain a metric on $K_X+S+B$ adapted to $u$, and then the extension of $u$ follows by our version of Ohsawa-Takegoshi 
(which is applied several times in the proof of \eqref{t-ext}). 
\smallskip

Theorem \ref{t-ext} will be discussed in more detail in Section \ref{s-fingen}.  
In many applications the following corollary to \eqref{t-ext} suffices.
\begin{corollary}\label{c-ext}  Let $K_X+S+B$ be a nef plt pair such that there exists an effective  $\mathbb Q$-divisor $D\sim _{\mathbb Q}K_X+S+B $  with $S\subset {\rm Supp }(D)\subset {\rm Supp }(S+B)$. 
Then $$H^0(X,\mathcal O _X(m(K_X+S+B)))\to H^0(S,  \mathcal O _S(m(K_X+S+B)))$$
is surjective for all sufficiently divisible integers $m>0$.

In particular, if $\kappa ((K_X+S+B)|_S)\geq 0$, then the stable base locus of $K_X+S+B$ does not contain $S$.
\end{corollary}
This paper is organized as follows:
In Section \ref{s-pre} we recall the necessary notation, conventions and preliminaries. In Section \ref{s-fingen} we give some background on the analytic approach and in particular we explain the significance of good minimal models in the analytic context. In Section \ref{s-OT} we prove a Ohsawa-Takegoshi extension theorem which generalizes a result of L. Manivel and D. Varolin. 
In Section \ref{s-pot} we prove the Extension Theorem \ref{t-ext}. Finally, 
in Section \ref{s-good} we prove Theorem \ref{t-good1}.
\section{Preliminaries}\label{s-pre}
\subsection{Notation and conventions}
We work over the field of complex numbers $\mathbb C$.

Let $D=\sum d_iD_i$ and $D'=\sum d_i'D_i$ be  $\mathbb Q$-divisors on a normal variety $X$, then the {\bf round-down} of $D$ is given by $\lfloor D \rfloor 
:=\sum \lfloor d_i \rfloor D _i$ where $\lfloor d_i \rfloor={\rm max}\{ z\in \mathbb Z | z\leq d_i \}$. Note that by definition we have $|D|=|\lfloor D \rfloor |$.
We let $D\wedge D':=\sum {\rm min}\{d_i,d_i' \}D_i$ and $D\vee D':=\sum {\rm max}\{d_i,d_i' \}D_i$.
The $\mathbb Q$-Cartier divisor $D$ is {\bf nef} if $D\cdot C\geq 0$ for any curve $C\subset X$.
The $\mathbb Q$-divisors $D$ and $D'$ are {\bf numerically equivalent} $D\equiv D'$ if and only if $(D-D')\cdot C=0$ for any curve $C\subset X$.
The {\bf Kodaira dimension} of $D$ is $$\kappa (D):={\rm tr.deg}_{\mathbb C}\left( \oplus _{m\geq 0}H^0(X,\mathcal O _X(mD))\right) -1.$$ If $\kappa (D)\geq 0$, then $\kappa (D)$ is the smallest integer $k>0$ such that
 $\liminf h^0(\mathcal O_X(mD))/{m^{k}}>0$. We have $\kappa (D)\in \{ -1,0, 1,\ldots , \dim X \}$. If $\kappa (D)=\dim X$ then we say that $D$ is {\bf big}. 
 If $D\equiv D'$ then $D$ is big if and only if $D'$ is big.
If $D$ is numerically equivalent to a limit of big divisors, then we say that $D$ is {\bf pseudo-effective}. 

Let $A$ be a sufficiently ample divisor, and $D$ is a pseudo-effective $\mathbb Q$-divisor, then we define $$\kappa _\sigma (D):={\rm max }\{k>0| \limsup_{m\to \infty} \frac{h^0(\mathcal O _X(mD+A))}{m^k}<+\infty \}.$$
It is known that $\kappa (D)\leq \kappa _\sigma (D)$ and equality holds when $\kappa _\sigma (D)=\dim X$.

Let $V\subset |D|$ be a linear series, then we let ${\rm Bs} (V)=\{x\in X| x\in {\rm Supp }(E)\ \forall\ E\in V \}$ be the {\bf base locus of} $V$ and ${\rm Fix}(V)=\wedge _{E\in V}E$ be the {\bf fixed part} of $|V|$. In particular $|V|=|V-F|+F$ where $F={\rm Fix}(V)$.
If $V_i\subset |iD|$ is a sequence of (non-empty) linear series such that $V_i\cdot V_j\subset V_{i+j}$ for all $i,j>0$, then we let ${\mathbf B}(V_\bullet )=\cap _{i>0}{\rm Bs}(B_i)$ be the {\bf stable base locus} of $V_\bullet$ and $\mathbf{ Fix}(V_\bullet )=\cap _{i>0}{\rm Supp}({\rm Fix}(B_i))$ be the {\bf stable fixed part} of $V_\bullet$.
When $V_i=|iD|$ and $\kappa (D)\geq 0$, we will simply write $\mathbf {Fix}(D)=\mathbf{ Fix}(V_\bullet )$ and ${\mathbf B}(D)={\mathbf B}(V_\bullet )$.
If $D$ is pseudo-effective and $A$ is an ample divisor on $X$, then we let 
${\mathbf B}_-(D)=\bigcup _{\epsilon \in \mathbb Q _{>0}}{\mathbf B}(B+\epsilon A )$
be the {\bf diminished stable base locus}.
If $C$ is a prime divisor, and $D$ is a big $\mathbb Q$-divisor, then
 we let $$\sigma _C(D)={\rm inf}\{ {\rm mult} _C (D')|D'\sim _{\mathbb Q}D,\ D'\geq 0 \},$$
and if $D$ is pseudo-effective then we let $$\sigma _C(D)=\lim _{\epsilon \to 0}\sigma _C(D+\epsilon A).$$
Note that $\sigma _C(D)$ is independent of the choice of $A$ and is determined by the numerical equivalence class of $D$. Moreover the set of prime divisors
for which $\sigma _C(D)\ne 0$ is finite (for this and other details about $\sigma _C(D)$, we refer the reader to \cite{Nakayama04}).
One also defines the $\mathbb R$-divisor $${N}_\sigma (D)=\sum _{C}\sigma _C(D)C$$ so that the support of ${N}_\sigma (D)$ equals $\bigcup _{\epsilon \in \mathbb Q _{>0}}\mathbf {Fix}(B+\epsilon A )$.

If $S$ is a normal prime divisor on a normal variety $X$, $P$ is a prime divisor on $S$ and $D$ is a divisor such that $S$ is not contained in ${\bf B}_+(D)$ then we define 
$$\sigma _P(||D||_S)={\rm inf }\{ {\rm mult }_P(D'|_S)|D'\sim _{\mathbb Q}D,\ D'\geq 0,\ S\not \subset {\rm Supp}(D') \}.$$
If instead $D$ is a pseudo-effective divisor such that $S\not \subset {\mathbf B}_-(D)$ then we let $$\sigma _P(||D||_S)=\lim _{\epsilon \to 0}\sigma _P(||D+\epsilon A||_S).$$
Note that $\sigma _P(||D||_S)$ is determined by the numerical equivalence class of $D$ and independent of the choice of the ample divisor $A$.
One can see that the set of prime divisors such that $\sigma _P(||D||_S)>0$ is countable. For this and other details regarding $\sigma _P(||D||_S)$ we refer the reader to Section 9 of \cite{HK10}.
We now define $N_\sigma (||D||_S)=\sum _P\sigma _P(||D||_S)P$. Note that $N_\sigma (||D||_S)$ is a formal sum of countably many prime divisors on $S$ with positive real coefficients.
\subsection{Singularities of the mmp}
If $X$ is a normal quasi-projective variety and $\Delta$ is an effective  $\mathbb Q$-divisor such $K_X+\Delta$ is $\mathbb Q$-Cartier, then we say that 
$(X,\Delta )$ is a {\bf pair}.
We say that a pair $(X,\Delta )$ is {\bf log smooth} if $X$ is smooth and the support of $\Delta$ has simple normal crossings.
A {\bf log resolution} of a pair $(X,\Delta )$ is a projective birational morphism $f:Y\to X$ such that $Y$ is smooth, the exceptional set ${\rm Exc}(f)$ is a divisor with simple normal crossings support and $f^{-1}_*\Delta +{\rm Exc}(f)$
has simple normal crossings support.
We will write $K_Y+\Gamma =f^*(K_X+\Delta )+E$ where $\Gamma $ and $E$ are effective with no common components.
We say that $(X,\Delta)$ is {\bf Kawamata log terminal} or {\bf klt} (resp. {\bf log canonical} or {\bf lc}) if there is a log resolution (equivalently for any log resolution)
of $(X,\Delta )$ such that the coefficients of $\Gamma $ are $<1$ (resp. $\leq 1$). We say that  $(X,\Delta)$ is {\bf divisorially log terminal} or {\bf dlt} if the coefficients of $\Delta$ are $\leq 1$ and there is a log resolution such that the coefficients of $\Gamma - f^{-1}_*\Delta$ are $<1$. In this case if we write $\Delta =S+B$ where $S=\sum S_i=\lfloor \Delta \rfloor$ then each component of a stratum $S_I=S_{i_1}\cap \ldots \cap S_{i_k}$ of $S$ is normal and $(S_I,\Delta _{S_I})$ is dlt where
$K_{S_I}+\Delta _{S_I}=(K_X+\Delta )|_{S_I}$.
If $(X,\Delta)$ is dlt and $S$ is a disjoint union of prime divisors, then we say that $(X,\Delta)$ is  {\bf purely log terminal} or {\bf plt}. This is equivalent to requiring that
$(S_i,\Delta _{S_i})$ is klt for all $i$. Often we will assume that $S$ is prime.

\subsection{The minimal model program with scaling}\label{s-mmps}
A proper birational map $\phi :X\dasharrow X'$ is a {\bf birational contraction} if $\phi ^{-1}$ contracts no divisors. 
Let $(X,\Delta )$ be a projective  $\mathbb Q$-factorial dlt pair and $\phi :X\dasharrow X'$ a birational contraction to a normal $\mathbb Q$-factorial variety $X'$, then $\phi$ is $K_X+\Delta$-{\bf negative} (resp. {\bf non-positive}) if
$a(E,X,\Delta)<a(E,X',\phi _* \Delta )$ (resp. $a(E,X,\Delta)\leq a(E,X',\phi _* \Delta )$)
for all $\phi$-exceptional divisors.
 If moreover
$K_{X'}+\phi _* \Delta$ is nef then $\phi$ is a   {\bf minimal model} for $(X,\Delta )$ (or equivalently a $K_X+\Delta$-minimal model).
Note that in this case (by the Negativity Lemma),
we have that $a(E,X,\Delta)<a(E,X',\phi _* \Delta )$ for all divisors $E$ over $X$ (resp. $a(E,X,\Delta)\leq a(E,X',\phi _* \Delta )$)
and $(X' , \phi _* \Delta)$ is dlt.
If moreover $K_X+\phi _* \Delta$ is semiample, then we say that $\phi$ is a 
 {\bf good minimal model} for $(X,\Delta )$.
Note that if $\phi$ is a 
good minimal model for $(X,\Delta )$, then \begin{equation}\label{e-ex}{\rm Supp}(\mathbf {Fix}(K_X+\Delta))={\rm Supp}(N_\sigma (K_X+\Delta))={\rm Exc}(\phi )\end{equation} is the set of $\phi$-exceptional divisors.
Another important remark is that if $\phi$ is a minimal model, then $H^0(X,\mathcal O _X(m(K_X+ \Delta )))\cong H^0(Y,\mathcal O _Y(m(K_Y+\phi _* \Delta )))$. More generally we have the following:
\begin{rem} \label{l-eq}If $\phi :X\dasharrow Y$ is a birational contraction such that $a(E,X,\Delta)\leq a(E,X',\phi _* \Delta )$ for every divisor $E$ over $X$, then
$$H^0(X,\mathcal O _X(m(K_X+ \Delta )))\cong H^0(Y,\mathcal O _Y(m(K_Y+\phi _* \Delta )))$$
for all $m>0$.
\end{rem}

Let $f:X\to Z$ be a proper morphism surjective with connected fibers from a $\mathbb Q$-factorial dlt pair $(X,\Delta )$ such that $\rho (X/Z)=1$ and $-(K_X+\Delta)$ is $f$-ample.
\begin{enumerate}
\item If $\dim Z<\dim X$, we say that $f$ is a {\bf Fano-Mori contraction}.
\item If $\dim Z=\dim X$ and $\dim {\rm Exc}(f)=\dim X -1$, we say that $f$ is a {\bf divisorial contraction}.
\item If $\dim Z=\dim X$ and $\dim {\rm Exc}(f)<\dim X -1$, we say that $f$ is a {\bf flipping   contraction}.\end{enumerate}
If $f$ is a divisorial contraction, then $(Z,f_*\Delta )$ is a $\mathbb Q$-factorial dlt pair.
If $f$ is a flipping contraction, then by \cite{BCHM10}, the flip $f^+:X^+\to Z$ exists, it is unique and given by $$X^+={\rm Proj}_Z\oplus _{m\geq 0}f_*\mathcal O _X(m(K_X+\Delta )).$$
We have that the induced rational map $\phi :X\dasharrow X^+$ is an isomorphism in codimension $1$ and $(X^+,\phi _* \Delta )$ is a $\mathbb Q$-factorial dlt pair.

Let $(X,\Delta )$ be a projective $\mathbb Q$-factorial dlt pair (resp. a klt pair), and $A$ an ample (resp. big) $\mathbb Q$-divisor such that $K_X+\Delta +A$ is nef.
By \cite{BCHM10}, we may run the {\bf minimal model program with scaling} of $A$, so that we get a sequence of birational contractions
$\phi _i:X_i\dasharrow X_{i+1}$  where $X_0=X$ and of rational numbers $t_i\geq t_{i+1}$ such that
\begin{enumerate}
\item if $\Delta _{i+1}:={\phi _i}_*\Delta _i$ and $H_{i+1}={\phi _i}_*H_i$, then
$(X_i, \Delta _i)$ is a $\mathbb Q$-factorial dlt pair (resp. a klt pair) for all $i\geq 0$,
\item $K_{X_i}+\Delta _i +tH_i$ is nef for any $t_i\geq t\geq t_{i+1}$,
\item if the sequence is finite, i.e. $i=0,1,\ldots , N$, then $K_{X_N}+\Delta _N +t_NH_N$ is nef or there exists a Fano-Mori contraction $X_N\to Z$,
\item if the sequence is infinite, then $\lim t_i=0$. 
\end{enumerate}
If the sequence is finite, we say that the minimal model program with scaling terminates. Conjecturally this is always the case.
\begin{rem}\label{r-s} Note that it is possible that $t_i=t_{i+1}$. Moreover it is known that there exist infinite sequences of flops (cf. \cite{Kawamata97}), i.e. $K_X+\Delta$ trivial maps.
\end{rem}
\begin{rem}\label{r-n} Note that if $K_X+\Delta$ is pseudo-effective then the support of $N_{\sigma}(K_X+\Delta )$ contains finitely many prime divisors and it coincides with the support of ${\bf Fix}(K_X+\Delta +\epsilon A)$ for any $0<\epsilon \ll 1$ (cf. \cite{Nakayama04}). It follows that if the sequence of flips with scaling is infinite, then $N_{\sigma}(K_{X_i}+\Delta _i)=0$ for all $i\gg 0$.
\end{rem}
\begin{theorem} \label{t-term}
If either \begin{enumerate}\item no component of $\lfloor S \rfloor $ is contained in ${\bf B}_-(K_X+\Delta )$ (eg. if $K_X+\Delta $ is big and klt), or 
\item $K_X+\Delta $ is not pseudo-effective, or 
\item $(X,\Delta )$ has a good minimal model,\end{enumerate} then the minimal model program with scaling terminates.\end{theorem}
\begin{proof} See \cite{BCHM10} and \cite{Lai10}.\end{proof}
\begin{rem}\label{r-rel} It is important to observe that in \cite{BCHM10} the above results are discussed in the relative setting.
In particular it is known that if $(X,\Delta )$ is a klt pair and $\pi:X\to Z$ is a birational projective morphism, then $(X,\Delta )$ has a good minimal model over $Z$.
More precisely there exists a finite sequence of flips and divisorial contractions over $Z$ giving rise to a birational contraction $\phi:X\dasharrow X'$ over $Z$ such that $K_{X'}+\phi _* \Delta$ is semiample over $Z$ (i.e. there is a projective morphism $q:X'\to W$ over $Z$ such that $K_{X'}+\phi _* \Delta\sim _{\mathbb Q}q^*A$ where $A$ is a $\mathbb Q$-divisor on $W$ which is ample over $Z$).
\end{rem}
\begin{rem}\label{r-good} It is known that the existence of good minimal models for pseudo-effective klt pairs is equivalent to the following conjecture (cf. \cite{GL10}): If $(X,\Delta )$ is a pseudo-effective klt pair, then  $\kappa _\sigma (K_X+\Delta )=\kappa (K_X +\Delta )$. 

Suppose in fact that $(X,\Delta )$ has a good minimal model say $(X',\Delta ')$ and let $f:X'\to Z={\rm Proj}R(K_{X'}+\Delta ')$ so that $K_{X'}+\Delta '=f^*A$ for some ample $\mathbb Q$-divisor $A$ on $Z$.
Following Chapter V of \cite{Nakayama04}, we have 
$$\kappa _\sigma (K_X+\Delta )=\kappa _\sigma (K_{X'}+\Delta ')=\dim Z=\kappa (K_{X'}+\Delta ')=\kappa (K_{X}+\Delta ).$$
Conversely, assume that $\kappa _\sigma (K_X+\Delta )=\kappa (K_X +\Delta )\geq 0$.
If $\kappa _\sigma (K_X+\Delta )=\dim X$, then the result follows from \cite{BCHM10}.
If $\kappa _\sigma (K_X+\Delta )=0$, then by \cite[V.1.11]{Nakayama04}, we have that $K_X+\Delta$ is numerically equivalent to $N_\sigma(K_X+\Delta)$.
By \cite{BCHM10} (cf. \eqref{r-n}), after finitely many steps of the minimal model program with scaling, we may assume that  $N_\sigma(K_{X'}+\Delta ')=0$ and hence that $K_{X'}+\Delta' \equiv 0$. Since $\kappa (K_X+\Delta) =0$, we conclude that $K_{X'}+\Delta ' \sim _{\mathbb Q}0$ and hence $(X,\Delta)$ has a good minimal model.
Assume now that $0<\kappa _\sigma (K_X+\Delta )<\dim X$.
Let $f:X\to Z={\rm Proj}R(K_{X}+\Delta )$ be a birational model of the Iitaka fibration with very general fiber $F$. By Chapter V of \cite{Nakayama04}, we have that $\kappa _\sigma (K_X+\Delta )=\kappa _\sigma (K_F+\Delta |_F)+\dim Z$, but since $\dim Z=\kappa (K_X+\Delta )$, we have that $\kappa _\sigma (K_F+\Delta |_F)=0$. Thus $(F,\Delta |_F)$ has a good minimal model.
By \cite{Lai10}, $(X,\Delta )$ has a good minimal model.
\end{rem}
Recall the following result due to Shokurov known as Special Termination:
\begin{theorem}\label{t-st} Assume that the minimal model program with scaling for klt pairs of dimension $\leq n-1$ terminates. 
Let $(X,\Delta )$ be a $\mathbb Q$-factorial $n$-dimensional dlt pair and $A$ an ample divisor such that $K_X+\Delta +A$ is nef.
If $ \phi _i:X_i\dasharrow X_{i+1}$ is a minimal model program with 
scaling, then $\phi _i$ is an isomorphism on a neighborhood of $\lfloor \Delta _i \rfloor$ for all $i\gg 0$.

If moreover $K_X+\Delta \equiv D\geq 0$ and the support of $D$ is 
contained in the support of $\lfloor \Delta \rfloor$ then the minimal model program with 
scaling terminates.\end{theorem}
\begin{proof} See \cite{Fujino07}.
\end{proof}
We will also need the following standard results about the minimal model program.
\begin{theorem}\label{t-er}[Length of extremal rays]
Let $(X,\Delta )$ be a lc and $(X,\Delta _0)$ be a klt pair and $f:X\to Z$ be a projective morphism surjective with connected fibers such that $\rho (X/Z)=1$ and $-(K_X+\Delta )$ is $f$-ample.

Then there exists a curve $\Sigma$ contracted by $F$ such that $$0<-(K_X+\Delta )\cdot \Sigma \leq 2\dim X.$$
\end{theorem}
\begin{proof} See for example \cite[3.8.1]{BCHM10}.
\end{proof}
\begin{theorem}\label{t-cone} Let $f:X\to Z$ be a flipping contraction and $\phi :X\dasharrow X^+$ be the corresponding flip. If $L$ is a nef and Cartier divisor such that $L\equiv _Z 0$, then so is $\phi _*L$.
\end{theorem}
\begin{proof} Easy consequence of the Cone Theorem, see for example \cite[3.7]{KM98}.\end{proof}

\subsection{A few analytic preliminaries}\label{sing_metrics}

We collect here some definitions and results concerning (singular) metrics on line bundles, which
will be used in the sections that follow. For a more detailed presentation and discussion,
we refer the reader to \cite{Dem90}.

\begin{defn} Let $L\to X$ be a line bundle on a compact complex manifold. A singular hermitian metric $h_L$ on 
$L$ is given in any trivialization $\theta: L|_{\Omega}\to \Omega\times {\C}$ by
$$\vert \xi\vert^2_{h_L}:= |\theta(\xi)|^2e^{-\varphi_L(x)}, \quad \xi \in L_x$$
where $\varphi_L\in L^1_{\rm loc}(\Omega)$  is the local weight of the metric 
$h_L$ and  $h_L= e^{-\varphi_L}$.
\end{defn}

The difference between the notions of smooth and singular metrics is that 
in the latter case the local weights are only assumed to verify a weak regularity property.
The hypothesis $\varphi_L\in L^1_{\rm loc}(\Omega)$ is needed in order to define the \emph{curvature current} of $(L, h_L)$, as follows:
$$\Theta_{h_L}(L)|_{\Omega}:= {\sqrt {-1}\over 2\pi}\ddbar \varphi_L.$$
If the local weights $\varphi_L$ of $h_L$ are \emph{plurisubharmonic}
(``psh" for short, see \cite{Dem90} and the references therein), then we have $\Theta_{h_L}(L)\geq 0$; conversely, if we know that if
$\Theta_{h_L}(L)\geq 0$, then each $\varphi_L$ coincides almost everywhere
with a psh function. 
\smallskip

We next state one of the important properties of the class of psh functions, which will be used several times in the proof of \eqref{t-ext}. Let $\beta$ be a $\mathcal C^\infty$-form
of (1, 1)-type, such that 
$d\beta= 0$. Let $\tau_1$ and $\tau_2$ be two functions in $L^1(X)$, such that
$$\beta+ \sqrt{-1}\ddbar \tau_j\geq 0$$
on $X$, for each $j= 1, 2$. We define $\displaystyle \tau:= \max(\tau_1, \tau_2)$,
and then we have
$$\beta+ \sqrt{-1}\ddbar \tau\geq 0$$
on $X$ (we refer e.g. to \cite{Dem09} for the proof).

\subsection{Examples}\label{ss-e}One of the best known and useful examples of singular metrics appears in the context of algebraic geometry: we assume that $L^{\otimes m}$ has some global holomorphic sections 
say $\{ \sigma_j\}_{j\in J}$. Then there is a metric on $L$, whose local weights can be described by
$$\varphi_L(x):= {1\over m}\log\sum_{j\in J}|f_j(x)|^2$$
where the holomorphic functions $\{ f_j\}_{j\in J} \subset {\mathcal O}(\Omega)$ are the local expressions of the global sections $\{ \sigma_j\}_{j\in J}$. The singularities of the metric defined
above are of course the common zeroes of $\{ \sigma_j\}_{j\in J}$. One very important property of these metrics is the \emph{semi-positivity of the curvature current}
$$\Theta_{h_L}(L)\geq 0,$$
as it is well known that the local weights induced by the sections $\{ \sigma_j\}_{j\in J}$ above are psh.
If the metric $h_L$ is induced by one section $\sigma\in H^0(X, L^{\otimes m})$ with zero set $Z_\sigma$, then we have that 
$$\Theta_{h_L}(L)= {1\over m}[Z_\sigma],$$
hence the curvature is given (up to a multiple) 
by the current of integration over the zero set of $\sigma$. From this point of view, the curvature of a singular hermitian metric is a natural generalization of  
an effective ${\Q}$-divisor
in algebraic geometry. 

\medskip

\noindent A slight variation on the previous example is the following. Let $L$ be a line bundle, which is \emph{numerically equivalent} to an effective $\Q$-divisor 
$$D= \sum_j\nu^jW_j.$$

Then $D$ and $L$ have the same 
first Chern class, hence there is an integer $m>0$ such that $L^{\otimes m}= \mathcal O_X(mD)\otimes \rho^{\otimes m}$, for some topologically trivial line bundle $\rho\in \Pic^0(X)$.

In particular, there exists a metric $h_\rho$ on the line bundle $\rho$ whose curvature is equal to zero (i.e. the local weights $\varphi_\rho$ of $h_\rho$ are real parts of holomorphic functions).
Then the expression
$$\varphi_\rho+ \sum_j\nu^j\log |f_j|^2$$
(where $f_j$ is the local equation of $W_j$) is the local weight of a metric on $L$; we call it \emph{the metric induced by} $D$
(although it depends on the choice of $h_\rho$).

\medskip

\noindent The following result is not strictly needed in this article, but we mention it because we feel that it may help to understand the structure of the curvature currents associated with singular metrics.
\begin{theorem}\label{t-curr} \cite {Siu74} Let $T$ be a closed positive current of $(1,1)$-type. Then we have
$$T= \sum_{j\geq 1}\nu^j[Y_j]+ \Lambda$$
where the $\nu^j$ are positive real numbers, and $\{ Y_j\}$ is a (countable) family
of hypersurfaces of $X$ and $\Lambda$ is a closed positive current whose singularities are concentrated along a countable union of analytic subsets of codimension at least two.
\end{theorem}
We will not make precise the notion of ``singularity" appearing in the statement above. 
We just mention that it is the analog of the multiplicity of a divisor.
By Theorem \ref{t-curr}
we infer that if the curvature current of a singular metric is positive, then it can be decomposed into a divisor-like part (however, notice that the sum above may be infinite), together with a \emph{diffuse} 
part $\Lambda$, which --very, very roughly--corresponds to a differential form. 
\smallskip

As we will see in Section \ref{s-OT} below, it is crucial to be able
to work with singular metrics in full generality: the hypothesis of all vanishing/extension theorems that we are aware of, are mainly concerned with the diffuse part of the curvature current, and not the singular one. Unless explicitly mentioned otherwise, all the metrics in this article
are allowed to be singular.
\bigskip

\subsection{Construction of metrics}We consider now the following set-up. Let $L$ be a 
$\Q$-line bundle, such that:

\begin{enumerate}
\item[(1)] $L$ admits a metric $h_L= e^{-\varphi_L}$ with positive curvature current $\Theta_{h_L}(L)$.

\item[(2)] The $\Q$-line bundle $L$ is numerically equivalent to the effective $\Q$-divisor
$$D:= \sum_{j\in J}\nu^jW_j$$ 
where $\nu^1> 0$ and the restriction of $h_L$ to the generic point of 
$W_1$ is well-defined (i.e. not equal to $\infty$). We denote by $h_D$ the metric on 
$L$ induced by the divisor $D$.

\item[(3)] Let $h_0$ be a non-singular metric on $L$; then we can write
$$h_L= e^{-\psi_1}h_0, \quad h_D= e^{-\psi_2}h_0$$ 
where 
$\psi_j$ are global functions on $X$. Suppose that we have
$$\psi_1\geq \psi_2.$$
Working locally on some coordinates 
open set $\Omega\subset X$, if we let 
$\varphi_L$
be the local weight of the metric $h_L$, and for each $j\in J$ we let $f_j$ be an equation of $W_j\cap \Omega$, then  the above inequality is equivalent to
\begin{equation}\tag{$\dagger$}\varphi_L\geq \varphi_\rho+ \sum_{j\in J}\nu^j\log |f_j|^2\end{equation}
(cf. the above discussion concerning the metric induced by a $\Q$-divisor
numerically equivalent to $L$). 
\end{enumerate}

\medskip

\noindent In this context, we have the following simple observation.

\begin{lemma}\label{l-weight} Let $\Omega\subset X$ be a coordinate open set.
Define functions $\varphi_{W_1}\in L^1_{\rm loc}(\Omega)$ which are the local weights of a metric on $\mathcal O_X(W_1)$, via the equality
$$\varphi_L= \nu^1\varphi_{W_1}+ \varphi_\rho+ \sum_{j\in J\setminus 1}\nu^j\log |f_j|^2.$$
Then $(\dagger )$ is equivalent to the inequality
$$|f_1|^2e^{-\varphi_{W_1}}\leq 1$$
at each point of $\Omega$.
\end{lemma}

\begin{proof}
It is a consequence of the fact that $\log |f_j|^2$ are the local weights of the singular metric on $\mathcal O_X(W_j)$ induced by the tautological section of this line bundle,
combined with the fact that $L$ and $\mathcal O_X(D)$ are numerically equivalent. The inequality above
is equivalent to $(\dagger)$.
\end{proof}

\bigskip

\subsection{Mean value inequality}We end this subsection by recalling a form of the mean value inequality for 
psh functions, which will be particularly useful in Section \ref{s-pot}.

Let $\alpha$ be a smooth $(1, 1)$ form on $X$, such that $d\alpha= 0$, and let $f\in L^1(X)$
be such that 
\begin{equation}\tag{+}\alpha +\sqrt {-1}\ddbar f\geq 0.\end{equation}
We fix the following quantity
$$ I(f):= \int_Xe^fdV_\omega$$
where $dV_\omega$ is the volume element induced by a metric $\omega$ on $X$. We have the following well-known result.

\begin{lemma}\label{l-C}
There exists a constant $C= C(X, \omega, \alpha)$ such that for any function 
$f\in L^1(X)$ verifying the condition $(+)$ above we have
$$f(x)\leq C+ \log I(f), \quad \forall x\in X.$$
\end{lemma}

\begin{proof}
We consider a coordinate system $z:= \{ z^1,\ldots , z^n\}$ defined on 
$\Omega\subset X$ and centered at some point $x\in X$. 
Let $B_r:= \{ \Vert z\Vert< r\}$ be the Euclidean ball of radius $r$, and let $d\lambda$ be
the Lebesgue measure corresponding to the given coordinate system. Since $X$ is a compact manifold, 
we may assume that the radius $r$ is independent of the particular point $x\in X$.

By definition of $I(f)$ we have
$$I(f)\geq {1\over \Vol(B_r)}\int_{z\in B_r}e^{f(z)+ C(X, \omega)}d\lambda$$
where $C(X, \omega)$ takes into account the distortion between 
the volume element $dV_\omega$ and the local Lebesgue measure $d\lambda$, together with the Euclidean volume of $B_r$.

We can assume the existence of a function $g_\alpha\in \mathcal C^\infty (\Omega)$
such that $\alpha|_{\Omega}= \sqrt{-1}\ddbar g_\alpha$. By (+), the function
$f+ g_\alpha$ is psh on $\Omega$. We now modify the inequality above as follows
$$I(f)\geq {1\over \Vol(B_r)}\int_{z\in B_r}e^{f(z)+ g_\alpha(z)+ C(X, \omega, \alpha)}d\lambda.$$
By the concavity of the logarithm, combined with the mean value inequality 
applied to $f+ g_\alpha$ we infer that
$$\log I(f)\geq f(x)- C(X, \omega, \alpha)$$
where the (new) constant $C(X, \omega, \alpha)$ only depends on the geometry of
$(X, \omega)$ and on a finite number of potentials $g_\alpha$ (because of the compactness of $X$). The proof of the lemma is therefore finished. A last remark is that the constant $``C(X, \omega, \alpha)"$ is uniform with respect to $\alpha$:
given $\delta> 0$, there exists a constant $C(X, \omega, \alpha, \delta)$ such that 
we can take $C(X, \omega, \alpha^\prime):= C(X, \omega, \alpha,\delta)$
for any closed $(1,1)$-form $\alpha^\prime$ such that $\Vert \alpha- \alpha^\prime\Vert< \delta$. 

\end{proof}

\section{Finite generation of modules}\label{s-fingen}

According to Remark \ref{r-good}, in order to establish the existence of good minimal 
models for pseudo-effective, klt pairs it suffices to show that 
$$\kappa _\sigma (K_X+\Delta )=\kappa (K_X +\Delta ).$$ 
In this section we will provide a direct argument for the equality above
in the case where $\Delta$ is big. Even if this result is well known to experts and implicit in some of the literature, our point of view is slightly different (see however \cite{CL10} for a related point of view), and 
it turns out to be very useful as a guiding principle for the arguments that we will invoke
in order to prove Theorem \ref{t-ext}. 
\smallskip

Let $X$ be a smooth, projective variety, and let $\Delta$ be a big ${\mathbb Q}$-divisor, 
such that $(X, \Delta)$ is klt. Analytically, this just means that $\Delta$ can be endowed with a  
metric $h_\Delta= e^{-\varphi_\Delta}$ whose associated curvature current dominates
a metric on $X$, and such that $e^{-\varphi_\Delta}\in L^1_{\rm loc}(X)$.
To be precise, what we really mean at this point is that the line bundle associated to
$d_0\Delta$ can be endowed with a metric whose curvature current is greater than a 
K\"ahler metric, and whose $d_0-{\rm th}$ root is $h_\Delta$.

Let $A\subset X$ be an ample divisor. We consider the following vector space
$${\mathcal M}:= \bigoplus _{m\in d_0{\mathbb N}}H^0\big(X, \mathcal O _X(m(K_X+ \Delta)+ A)\big)$$
which is an ${\mathcal R}$-module, where
$\displaystyle {\mathcal R}:= \bigoplus _{m\in d_0{\mathbb N}}H^0\big(X, \mathcal O _X(m(K_X+ \Delta))\big).$

\medskip

\noindent In this section we will discuss the following result.

\begin{proposition}\label{p-fgm} ${\mathcal M}$ is a finitely generated 
${\mathcal R}$-module.

\end{proposition}

\medskip

\noindent Since the choice of the ample divisor $A$ is arbitrary, the above proposition implies that
we have $\kappa _\sigma (K_X+\Delta )=\kappa (K_X +\Delta )$.
\smallskip

\noindent We provide a sketch of the proof \eqref{p-fgm} below. As we have already mentioned,
the techniques are well-known, so we will mainly highlight the features relevant to our arguments. 
The main ingredients are the finite generation of ${\mathcal R}$, coupled with the extension techniques originated in \cite{Siu98} and Skoda's division theorem \cite{Skoda}.

\medskip
\noindent
\begin{proof}[Sketch of the proof of Proposition \ref{p-fgm}]
We start with some reductions; in the first place, we may assume that $\kappa _\sigma (K_X+\Delta )\geq 0$ and hence that $\kappa (K_X+\Delta )\geq 0$ cf. \cite{BCHM10}.
Next, we can assume that:

\noindent $\bullet$ There exists a finite
set of normal crossing hypersurfaces $\{ Y_j\}_{j\in J}$ of $X$, such that
\begin{equation}\label{5}\Delta= \sum_{j\in J}\nu^jY_j+ A_\Delta\end{equation}
where $0\leq \nu^j< 1$ for any $j\in J$, and $A_\Delta$ is an ample ${\mathbb Q}$-divisor.
This can be easily achieved on a modification of $X$.

\smallskip
\noindent $\bullet$ Since the algebra ${\mathcal R}$ is generated by
a finite number of elements, we may assume that it is generated  by the sections of $m_0(K_X+ \Delta)$,
where $m_0$ is sufficiently large and divisible. The corresponding metric of $K_X+ \Delta$ (induced by the generators of $\mathcal R$) is
denoted by $h_{\rm min}= e^{-\varphi_{\rm min}}$ (the construction was recalled in the Subsection \ref{ss-e}; see \cite{Dem09} for a more detailed presentation). Hence, we may assume that 
\begin{equation}\label{6}\Theta_{h_{\rm min}}(K_X+ \Delta)= \sum a^j_{\rm min}[Y_j]+
\Lambda_{\rm min}\end{equation} 
(after possibly replacing $X$ by a further modification). 
In relation \eqref{6} above, we can
take the set of $\{ Y_j\}_{j\in J}$ to coincide with the one in \eqref{5} (this is why we must allow some of the coefficients $\nu^j, a^j_{\rm min}$ to be equal to zero). 
$\Lambda_{\rm min}$ denotes a non-singular, semi-positive $(1,1)$-form.

\medskip

\noindent For each integer $m$ divisible enough, let $\displaystyle \Theta_m$ be the current induced by (the normalization of) a basis of sections of the divisor
$m(K_X+ \Delta)+ A$; it belongs to the cohomology class associated to
$\displaystyle K_X+ \Delta+ {1\over m}A$. We can decompose it according to the family of hypersurfaces
$\{ Y_j\}_{j\in J}$ as follows
\begin{equation}\label{7}\Theta_m= \sum_{j\in J}a_m^j[Y_j]+ \Lambda_m\end{equation}
where $a^j_m\geq 0$, and $\Lambda_m$ is a closed positive current, which may be singular, despite the fact that $\Theta_m$ is less singular than $\Theta_{h_{\rm min}}(K_X+ \Delta)$. Note that $a_m^j\leq a^j_{\rm min}$.
\medskip

\noindent An important step in the proof of Proposition \ref{p-fgm} is the following statement.

\begin{claim}\label{c-lim} We have
\begin{equation}\label{8}\lim_{m\to \infty}a^j_m= a^j_{\rm min}\end{equation}
for each $j\in J$ (and thus $\Lambda_m$ converges weakly to $\Lambda_{\rm min}$).

\end{claim}

\begin{proof}
\noindent We consider an element $j\in J$. The sequence $\{ a^j_m\}_{m\geq 1}$ is bounded, and can be assumed to be convergent, so we denote by $a^j_\infty$ its limit. We observe that we have
\begin{equation}\label{9}a^j_\infty\leq a^j_{\rm min}\end{equation}
for each index $j$. Arguing by contradiction, we assume that at least one of the inequalities \eqref{9} above is strict.

Let $\Lambda_\infty$ be any weak limit of the sequence $\{ \Lambda_m\}_{m\geq 1}$.
We note that in principle
$\Lambda_\infty$ {\sl will be singular} along some of the $Y_j$, even if the Lelong number of each $\Lambda_m$ at the generic point of $Y_j$ is equal to zero, for any
$j$: the reason is that any weak limit of $\{ \Theta_m\}_{m\geq 1}$ it is expected to be at least as singular as
$\Theta_{\rm min}$.

In any case, we remark that given any positive real number $t\in {\mathbb R}$, we have the following numerical identity
\begin{equation}\label{10}K_X+ \sum_{j\in J}\big(\nu^j+ t(a^j_{\rm min}- a^j_{\infty})- a^j_{\infty}\big)Y_j+ A_\Delta+ t\Lambda_{\rm min}\equiv
(1+ t)\Lambda_\infty.\end{equation}
By using the positivity of $A_\Delta$ to tie-break, we can assume that for all $j\in J$ such that $a^j_{\rm min}\ne a^j_{\infty}$, the quantities
\begin{equation}\label{11}t^j:= {1-\nu^j+ a^j_\infty\over a^j_{\rm min}- a^j_{\infty}}\end{equation}
 are distinct, and we moreover assume that the minimum is achieved for $j= 1$.
The relation \eqref{10} with $t= t^1$ becomes
\begin{equation}\label{12}K_X+ Y_1+ \sum_{j\in J, j\neq 1}\tau^jY_j+ A_\Delta\equiv
(1+ t^1)\Lambda_\infty\end{equation}
where $\tau^j:= \nu^j+ t^1(a^j_{\rm min}- a^j_{\infty})- a^j_{\infty}< 1$ are real numbers,
which can be assumed to be positive (since we can ``move" the negative 
ones on the right hand side).
But then we have the following result (implicit in \cite{Paun08}).

\begin{theorem}\label{t-3} There exists
an effective ${\mathbb R}$-divisor
$$D:= a^1_\infty Y_1+ \Xi$$
linearly equivalent to $K_X+ \Delta$ and such that $Y_1$ does not belong to the support of $\Xi$.
\end{theorem}
\noindent We will not reproduce here the complete argument of the proof, instead we highlight the main steps of this proof.
\smallskip

\noindent $\bullet$ Passing to a modification of $X$, we can assume that the hypersurfaces $\{ Y_j\}_{j\neq 1}$ are mutually disjoint and $A_\Delta$ is semi-positive (instead of ample), such that
$A_\Delta-\sum_{j\neq 1} \varepsilon^jY_j$ is ample (for some $0<\varepsilon^j\ll 1$ where the  corresponding $Y_j$ are exceptional divisors). We denote
$S:= Y_1$.
\smallskip

\noindent $\bullet$ We can assume that $[\Lambda_m]= [\Lambda_\infty]$
i.e. the cohomology class is the same
 for any $m$, but for ``the new" $\Lambda_m$ we only have
\begin{equation}\label{13}\Lambda_m\geq -{1\over m}\omega_{A_\Delta}.\end{equation}
\smallskip

\noindent $\bullet$ The restriction $\displaystyle \Lambda_{m}|_S$ is well-defined,
and can be written as
$$\Lambda_{m}|_S= \sum_{j\neq 1}\rho^j_mY_{j}|_S+ \Lambda_{m, S}.$$
\smallskip

\noindent $\bullet$ By induction, we obtain an effective ${\mathbb R}$-divisor $D_S$ {\sl linearly equivalent to} $(K_X+ S+ \sum_{j\in J, j\neq 1}\tau^jY_j+ A_{\Delta})|_S$, whose order of vanishing along $Y_j|_S$ is at least $\min \{ \tau^j, \rho^j_{\infty}\}$. We note that in \cite{Paun08}, we only obtain an effective ${\mathbb R}$-divisor $D_S$ which is {\sl numerically equivalent to}
$(K_X+ S+ \sum_{j\in J, j\neq 1}\tau^jY_j+ A_{\Delta})|_S$. By a standard argument, we may however assume that $D_S$ is $\mathbb R$-{\sl linearly equivalent to}
$(K_X+ S+ \sum_{j\in J, j\neq 1}\tau^jY_j+ A_{\Delta})|_S$ (see also \cite{CL10}).
\smallskip

\noindent $\bullet$ The ${\mathbb R}$-divisor $D_S$ extends to $X$, by the ``usual" procedure, namely Diophantine approximation and extension theorems
(see e.g. \cite{Paun08} and \cite{HM10}). This last step ends the discussion of the proof of Theorem \ref{t-3}.

\medskip

\begin {rem} The last bullet above is the heart of the proof of the claim. We stress the fact that the factor $A_{\Delta}$ of the boundary is essential, even if the coefficients $ \tau^j$ and the divisor $D_S$
are rational.

\end{rem}

\noindent An immediate Diophantine approximation argument 
shows that the divisor $D$ produced by Theorem \ref{t-3} should not exist: its 
multiplicity along $Y_1$ is \emph{strictly smaller than} $a^1_{\rm min}$, and this is a contradiction.
\end{proof}

\medskip

\noindent The rest of the proof is based of the following global version of the
H.~Skoda division theorem (cf. \cite{Skoda}), established in \cite{Siu08}.

Let $G$ be a divisor on $X$, and let $\sigma_1,\ldots , \sigma_N$ be a set of holomorphic sections of $\mathcal O_X(G)$. Let $E$ be a divisor on $X$, endowed with a possibly singular metric
$h_E= e^{-\varphi_E}$ with positive curvature current.

\begin {theorem} \cite{Skoda} Let $u$ be a holomorphic section of the divisor
$K_X+ (n+1)G+ E$, such that
\begin{equation}\label{15}\int_X{|u|^2e^{-\varphi_E}\over \big(\sum_j |\sigma_j|^2\big)^{n+1}}< \infty\end{equation}
(we notice that the quantity under the integral sign is a global measure on $X$).
Then there exists sections $u^1, \ldots , u^N$ of $K_X+ nG+ E$
such that 
\begin{equation}\label{16}u= \sum_ju^j\sigma_j.\end{equation}
\end{theorem}

\noindent This result together with Claim \ref{c-lim} above prove the finite generation of
${\mathcal M}$, along the same lines as in \cite{Dem90}; we provide next the details.
\smallskip

Let $m$ be a sufficiently big and divisible integer (to be specified in a moment), and let $u$
be a section of 
$m(K_X+ \Delta)+ A.$
We recall that $m_0$ denotes a positive integer, such that the metric on $K_X+ \Delta$ induced by
the sections $\{ \sigma_1,\ldots , \sigma _N\}$ of $m_0(K_X+\Delta)$ is equivalent to $\varphi_{\rm min}$.
We have
$$
m(K_X+ \Delta)+ A= K_X+ (n+1)G+ E$$
where 
$$G:= m_0(K_X+ \Delta)$$ and
$$E:= \Delta+ \big(m- (n+1)m_0-1\big)\Big(K_X+ \Delta+ {1\over m}A\Big)+ {m_0(n+1)+ 1\over m}A$$
are endowed respectively with the metrics 
$\displaystyle \varphi_G$ induced by the sections $\{ \sigma_1,\ldots , \sigma _N\}$ above, and
$$\displaystyle \varphi_E:= \varphi_\Delta+ \big(m- (n+1)m_0-1\big)\varphi_m+
{m_0(n+1)+1\over m}\varphi_A.$$ Here we denote by $\varphi_m$
the metric on $K_X+ \Delta+ {1\over m}A$ induced by the 
global sections of $\mathcal O _X(m(K_X+\Delta)+ A)$. 
We next check that condition \eqref{15} is satisfied. Notice that 
$$\int_X{|u|^2\over (\sum_j |\sigma_j|^2)^{n+1}}e^{-\varphi_E}\leq 
C\int_Xe^{(m_0(n+1)+1)\varphi_m- (n+1)m_0\varphi_{\rm min}-\varphi_\Delta}$$
since we clearly have $|u|^2\leq Ce^{m\varphi_m}$ (we skip the non-singular 
weight corresponding to $A$ in the expression above).
The fact that $(X, \Delta)$ is klt, together with Claim \ref{c-lim} implies that there exists some fixed index $m_1$ such that we have
\begin{equation}\label{17}\int_Xe^{(m_0(n+1)+1)\varphi_m- (n+1)m_0\varphi_{\rm min}-\varphi_\Delta}d\lambda< \infty,\end{equation}
as soon as $m\geq m_1$.
In conclusion, the relation \eqref{15} above holds true; hence, as long as $m\geq m_1$,
Skoda's Division Theorem can be applied, and Proposition \ref{p-fgm} is proved.
\end{proof}


\medskip

\begin{rem}\label{r-set}{\rm  As we have already mentioned, in the following sections we will show that a consistent part of the proof of \eqref{p-fgm} is still valid in
the absence of the ample part $A_\Delta$. Here we highlight the properties of
$K_X+ \Delta$ which will replace the strict positivity. We consider the following context.
Let
\begin{equation}\label{19}\Delta\equiv \sum_{j\in J}\nu^j[Y_j]+ \Lambda_\Delta\end{equation}
be a ${\mathbb Q}$-divisor, where $0\leq \nu^j< 1$ and $\Lambda_\Delta$ is a
semi-positive form of $(1,1)$-type. We assume as always that the hypersurfaces
$ Y_j$ have simple normal crossings. The difference between this set-up and the hypothesis
of \eqref{p-fgm} is that
$\Delta$ is not necessarily big.

We assume that the $K_X+ \Delta$ is ${\mathbb Q}$-effective. Recall that by \cite{BCHM10}, the associated canonical ring $R(K_X+ \Delta)$ is finitely generated. The reductions performed at the beginning of the proof of \eqref{p-fgm} do not use $A_\Delta$. However, difficulties arise when we come to the proof of Claim \ref{c-lim}. Indeed, the assumption
$$a^j_\infty< a^j_{\rm min}$$
for some $j\in J$ implies that we will have
\begin{equation}\label{20}K_X+ Y_1+ \sum_{j\in J, j\neq 1}\tau^jY_j\equiv
(1+ t^1)\Lambda_\infty\end{equation}
cf. \eqref{12}, but in the present context, the numbers $\tau^j$ {\sl cannot be assumed to be strictly smaller than $1$}. Nevertheless we have that 

\begin{enumerate}
\item[\rm (a)] The $\Q$-divisor $K_X+ Y_1+ \sum_{j\in J, j\neq 1}\tau^jY_j$ is pseudo-effective, and
$${Y_1}\not\in N_\sigma(K_X+ Y_1+ \sum_{j\in J, j\neq 1}\tau^jY_j).$$
\smallskip
\item[\rm (b)] There exists an effective ${\mathbb R}$-divisor say $G:= \sum_i\mu^iW_i$
which is ${\mathbb R}$-linearly equivalent to $K_X+ Y_1+ \sum_{j\in J, j\neq 1}\tau^jY_j$
such that 
$$Y_1\subset {\rm Supp}(G) \subset \{ Y_j\}_{j\in J}.$$
\end{enumerate}

}

\end{rem}

\noindent The properties above are consequences of the fact that
we have assumed that Claim \ref{c-lim} fails to hold. They indicate that the $\Q$-divisor
$$L:= K_X+ Y_1+ \sum_{j\in J, j\neq 1}\tau^jY_j$$ has some kind of positivity:  property
(a) implies the existence of a sequence of metrics $h_m= e^{-\varphi_m}$ on the $\Q$-line bundle $L$ such that
\begin{equation}\label{21}\Theta_{h_m}(L)= \sqrt{-1}\ddbar \varphi_m\geq -{1\over m}\omega\end{equation}
and this combined with (b) shows that the line bundle $\cO _X(Y_1)$ admits a sequence of singular metrics
$g_m:= e^{-\psi_{1, m}}$ such that
\begin{equation}\label{22}\varphi_m= \mu^1\psi_{1, m}+ \sum_{i\neq 1}\mu^i\log|f_{W_i}|^2,\end{equation}
where we assume that $W_1= Y_1$
(see Lemma \ref{l-weight}). So the curvature of $(L, h_m)$ is not just
bounded from below by $\displaystyle -{1\over m}\omega$, but we also have
\begin{equation}\label{23}\Theta_{h_m}(L)\geq \mu^1\Theta_{g_m}(Y_1)\end{equation}
 as shown by \eqref{22} above.
 \medskip

\noindent The important remark is that the relations \eqref{21} and \eqref{23} are very similar to the curvature requirement in the geometric version of the 
Ohsawa-Takegoshi-type theorem, due to L. Manivel (cf. \cite{Man93}, see \cite{Dem09} as well).
During the following section, we will establish the relevant generalization. As for the 
\emph{tie-breaking} issue (cf.  \eqref{20}), we are unable to bypass it with purely analytic methods. It will be treated by a different technique in Section \ref{s-good}.

\medskip

\section{A version of the Ohsawa-Takegoshi extension theorem}\label{s-OT}

\bigskip

\noindent The main building block of the proof of the ``invariance of plurigenera" (cf. \cite{Siu98}, \cite{Siu00})  is given by the Ohsawa-Takegoshi theorem cf. \cite{OT87} and \cite{BoB96}. In this section, we will prove a version of this important extension theorem, which will play a fundamental role
in the proof of Theorem \ref{t-ext}.
\smallskip

Actually, our result is a slight generalization of the corresponding statements 
in the articles quoted above,
adapted to the set-up described in Remark \ref{r-set}. For clarity of exposition, 
we will change the notations as follows.

Let $X$ be a projective manifold, and let $Y \subset X$ be a non-singular
hypersurface. We assume that there exists a metric
$h_{Y}$ on the line bundle ${\mathcal O}_X(Y)$ associated to $Y$,
denoted by $h_{Y}= e^{-\varphi_{Y}}$ with respect to any local 
trivialization, such that:

\begin{enumerate}

\item [(i)] If we denote by $s$ the tautological section associated to $Y$, then
\begin{equation}\label{24}|s|^2e^{-\varphi_{Y}}\leq e^{-\alpha}\end{equation}
where $\alpha\geq 1$ is a real number.

\item [(ii)] There exist two semi-positively curved hermitian $\bQ$-line bundles,
say $(G_1, e^{-\varphi_{G_1}})$ and $ (G_2, e^{-\varphi_{G_2}})$, such that
\begin{equation}\label{25}\varphi_{Y}= \varphi_{G_1}- \varphi_{G_2}\end{equation}
(compare to \eqref{22} above).

\end{enumerate}

\medskip
\noindent
Let $F\to
X$ be a line bundle, endowed with a
metric $h_F$ such that the following curvature requirements are satisfied
\begin{equation}\label{26}\Theta_{h_F}(F)\geq 0, \quad \Theta_{h_F}(F)\geq {1\over \alpha}
\Theta_{h_{Y}}(Y).\end{equation}
Moreover, we assume the existence of
positive real numbers $\varepsilon_0> 0$ and $C$ such that
\begin{equation}\label{27}\varphi_F\leq \varepsilon_0\varphi_{G_2}+C;\end{equation}
that is to say, the poles of the
metric which has the ``wrong'' sign in the decomposition \eqref{25} are part of the singularities
of $h_F$.
\medskip

\noindent We denote by $\ol h_{Y}= e^{-\ol\varphi_Y}$ a non-singular metric on the line bundle corresponding to $Y$. We have the following result.

\begin{theorem}\label{t-OT}Let $u$ be a section of the line bundle $\mathcal O _Y(K_{Y}+ F|_{Y})$, such that
\begin{equation}\label{28}\int_{Y}|u|^2e^{-\varphi_F}< \infty,\end{equation}
and such that the hypothesis \eqref{24}$\,$--$\,$\eqref{27} are satisfied.
Then there exists a section $U$ of the line bundle $\mathcal O _X(K_{X}+ {Y}+ F)$, such that
$U|_{Y}= u\wedge ds$ and such that
\begin{equation}\label{29}\int_{X}|U|^2e^{-\delta\varphi_{Y}-(1-\delta)\ol \varphi_{Y}-\varphi_F}\leq C_\delta
\int_{Y}|u|^2e^{-\varphi_F}\end{equation}
where $1\geq \delta> 0$ is an arbitrary real number and
the constant $C_\delta$ is given explicitly by
\begin{equation}
C_\delta=C_0\delta^{-2}\big( \max_X|s|^2e^{-\ol\varphi_Y}\big)^{1-\delta}
\end{equation}
for some numerical constant $C_0$ depending only on the dimension $($in 
particular, the estimate does not depend on $\varepsilon_0$ or $C$ 
in~\eqref{27}$)$.

\end{theorem}

\noindent Perhaps the closest statement of this kind in the literature is due to 
D. Varolin,
cf. \cite{Var08}; in his article, the  metric $h_{Y}$ is allowed to be singular,  but the weights of this metric are assumed to be {\sl bounded from above}. This hypothesis is not verified in our case; however, the assumption \eqref{27} plays a similar role in the proof of Theorem \ref{t-OT}.
\medskip

\begin{proof} We will closely follow the ``classical" arguments and show that the
proof goes through (with a few standard
modifications) in the more general setting of Theorem \ref{t-OT}.
The main issue which we have to address is  the regularization procedure.
Although the technique is more or less standard, since this is the key new ingredient, we will provide a complete treatment.

\subsection{Regularization procedure}
Let us first observe that every line bundle $B$ over $X$ can be written as a 
difference $B={\mathcal O}_X(H_1-H_2)$ of two very ample divisors 
$H_1$, $H_2$. It follows that $B$ is trivial 
upon restriction to the complement $X\setminus(H'_1\cup H'_2)$ for
any members $H'_1\in|H_1|$ and $H'_2\in|H_2|$ of the corresponding linear systems.
Therefore, one can find a finite family  $H_j\subset X$ of very 
ample divisors, 
such that $Y\not\subset H_j$ and each of the line bundles
under consideration $F$, ${\mathcal O}_X(Y)$ and $G_i^{\otimes N}$
(choosing $N$ divisible enough so that $G_i^{\otimes N}\in\Pic(X)$) is trivial on
the affine Zariski open set $X\setminus H$, where
$H=\bigcup H_j$. We also fix a proper embedding 
$X\setminus H\subset \bC^m$ in order to regularize 
the weights $\varphi_F$ and~$\varphi_{Y}$ of our metrics
on $X\setminus H$. The $L^2$ estimate will
be used afterwards to extend the sections to $X$ itself.

The arguments which follow are first carried out on a fixed affine open set 
$X\setminus H$ selected as above. In this respect, estimate \eqref{27} is 
then to be understood as valid only with a uniform constant $C=C(\Omega)$
on every relatively compact open subset $\Omega\subset\!\subset 
X\setminus H$. In order to regularize all of our weights 
$\varphi_F$ and $\varphi_Y=\varphi_{G_1}-\varphi_{G_2}$ respectively,
we invoke the following well known result which enables us
to employ the usual convolution kernel in Euclidean space.

\begin{theorem} \cite{Siu76} Given a Stein submanifold $V$ of a complex analytic space
$M$, there exist an open, Stein neighborhood $W\supset V$ of $V$, together
with a holomorphic retract $r: W\to V$.
\end{theorem}

\smallskip

\noindent In our setting, the above theorem shows the existence of
a Stein open set $W\subset \bC^m$, such that  $X\setminus H\subset W$, together with a holomorphic retraction $r:W\to X\setminus H$.
We use the map $r$ in order to extend the objects we have constructed on $X\setminus H$; we define
\begin{equation}\label{30}\wt \varphi_F:= \varphi_F\circ r, \quad \wt \varphi_{G_i}:= \varphi_{G_i}\circ r,\quad \wt \varphi_{Y}:= \wt \varphi_{G_1}- \wt \varphi_{G_2}.
\end{equation}
Next we will use a standard convolution kernel
in order to regularize the functions above.
We consider exhaustions
\begin{equation}\label{31}W= \bigcup _kW_k,\qquad X\setminus H=\bigcup X_k
\end{equation}
of $W$ by {\sl bounded} Stein domains (resp.\ of $X\setminus H$
by the relatively compact Stein open subsets $X_k:=(X\setminus H)\cap W_k$).
Let
$$\varphi_{F, \varepsilon}:=\wt \varphi_F*\rho_\varepsilon: W_k\to \bR$$
be the regularization of $\wt \varphi_F$, where for each $k$ we assume that $\varepsilon\leq \varepsilon (k)$
is small enough, so that the function above is well-defined. 
We use similar notations for the regularization of the other functions involved in the picture.
\smallskip

\noindent We show next that the normalization and curvature properties of the functions above
are preserved by the regularization process. In first place, 
the assumption $\Theta_{h_F}(F)\geq 0$ means that $\varphi_F$ is psh, hence
$\varphi_{F, \varepsilon}\geq\varphi_F$ and we still have
\begin{equation}\tag{\ref{24}$_\varepsilon$}|s|^2e^{-\varphi_{Y, \varepsilon}(z)}\leq 
|s|^2e^{-\varphi_{Y}(z)}\leq e^{-\alpha}.\end{equation}
on $X_k$. (Here of course $|s|$ means the absolute value of the section
$s$ viewed as a complex valued function according to the trivialization
of $\mathcal O _X(Y)$ on $X\setminus H$). Further, \eqref{25}-\eqref{26} implies that all functions
\begin{equation}\tag{\ref{26}$_\varepsilon$}z\to \varphi_{F, \varepsilon}(z), \quad z\to \varphi_{G_i, \varepsilon}(z), \quad
z\to \varphi_{F, \varepsilon}(z)- {1\over \alpha}\varphi_{Y, \varepsilon}(z)\end{equation}
are psh on $X_k$, by stability of plurisubharmonicity under convolution.
Finally, \eqref{27} leads to
\begin{equation}\tag{\ref{27}$_\varepsilon$}\varphi_{F, \varepsilon}\leq\varepsilon_0\varphi_{G_2, \varepsilon}+C(k)
\quad\hbox{on $W_k$,}\end{equation}
by linearity and monotonicity of convolution.
\smallskip

\noindent In conclusion, the hypothesis of \eqref{t-OT} are preserved by the particular regularization
process we have described here. We show in the following subsection that the ``usual" Ohsawa-Takegoshi theorem
applied to the regularized weights allows us to conclude.

\subsection{End of the proof of Theorem \ref{t-OT}}

We view here the section $u$ of $\mathcal O _Y(K_Y+F|_{Y})$ as a $(n-1)$-form on
$Y$ with values in $F|_{Y}$. Since $F$ is trivial on $X\setminus H$,
we can even consider $u$ as a complex valued $(n-1)$-form on
$Y\cap(X\setminus H)$.  The main result used in the proof of \eqref{t-OT} is
the following technical version of the Ohsawa-Takegoshi theorem.

\begin{theorem} \cite{Dem09} Let $M$ be a weakly pseudoconvex $n$-dimen\-sional manifold, and let 
$f: M\to \bC$ be a holomorphic function, such that $\partial f\neq 0$ on $f=0$. Consider two smooth 
functions $\varphi$ and $\rho$ on $M$, such that 
$$\sqrt{-1}\ddbar \varphi\geq 0, \quad \sqrt{-1}\ddbar \varphi\geq {1\over \alpha} \sqrt{-1}\ddbar \rho$$
and such that $\displaystyle |f|_\rho^2:= |f|^2e^{-\rho}\leq e^{-\alpha}$, where $\alpha\geq 1$ is a constant.
Then given a $n-1$ form $\gamma$ on $M_f:= \{f=0\}$, there exists a $n$-form $\Gamma$ on $M$ such that 
\item{\rm (a)}  $\displaystyle \Gamma|_{M_f}= \gamma\wedge df$;

\item{\rm (b)} We have
$$\int_M{|\Gamma|^2e^{-\rho- \varphi}\over |f|_\rho^2\log^2|f|_\rho^2}\leq C_0\int_{M_f}|\gamma|^2e^{-\varphi}$$
where $C_0$ is a numerical constant depending only on the dimension.

\end{theorem}

\smallskip

\noindent We apply the above version of the Ohsawa-Takegoshi theorem in our setting:
for each $k$ and for each $\varepsilon\leq \varepsilon(k)$ there exists a holomorphic $n$-form $U_{k, \varepsilon}$ on the Stein manifold $X_k$, such that
\begin{equation}\label{32}\int_{X_k}{|U_{k, \varepsilon}|^2e^{-\varphi_{Y, \varepsilon}- \varphi_{F,\varepsilon}}\over |s|^2e^{-\varphi_{Y, \varepsilon}}
\log^2\big(|s|^2e^{-\varphi_{Y, \varepsilon}}\big)}\leq C_0
\int_{Y\cap X_k}|u|^2e^{-\varphi_{F, \varepsilon}}\end{equation}
and such that $\displaystyle U_{k, \varepsilon}|_{Y\cap X_k}=u\wedge ds$.
Notice that 
$\varphi_{F, \varepsilon}\geq \wt \varphi_{F}= \varphi_F$ 
on $Y\cap X_k$, hence we get the $(\varepsilon,k)$-uniform upper bound
\begin{equation}\label{33}\int_{Y\cap X_k}|u|^2e^{-\varphi_{F, \varepsilon}}\leq
\int_{Y\cap(X\setminus H)}|u|^2e^{-\varphi_F}.\end{equation}
Our next task is to take the limit for $\varepsilon\to 0$ in the relation \eqref{32},
while keeping $k$ fixed at first. To this end, an important observation
is that
\begin{equation}\label{34}\int_{X_k}|U_{k, \varepsilon}|^2e^{-\delta \varphi_{Y, \varepsilon}-(1-\delta)\ol\varphi_Y-\varphi_{F,\varepsilon}}\leq C_\delta\int_{Y\cap(X\setminus H)}|u|^2e^{-\varphi_F}\end{equation}
for any $0<\delta\leq 1$. Indeed, the function
$t\to t^\delta\log^2(t)$ is bounded from above by $e^{-2}(2/\delta)^2\le \delta^{-2}$ when $t$ belongs to the fixed interval $[0, e^{-\alpha}]\subset[0,e^{-1}]$, 
so that $t(\log t)^2\le \delta^{-2}t^{1-\delta}$ and hence
\begin{equation}\label{35}{e^{-\varphi_{Y, \varepsilon}- \varphi_{F,\varepsilon}}\over |s|^2e^{-\varphi_{Y, \varepsilon}}
\log^2\big(|s|^2e^{-\varphi_{Y, \varepsilon}}\big)}\geq \delta^2
{e^{-\varphi_{Y, \varepsilon}- \varphi_{F,\varepsilon}}\over 
|s|^{2(1-\delta)}e^{-(1-\delta)\varphi_{Y, \varepsilon}}}.\end{equation}
We have used the uniform bound (\ref{24}$_\varepsilon$). We further
observe that by compactness of $X$ the continuous function $z\to |s(z)|^{2(1-\delta)}e^{-(1-\delta)\ol\varphi_Y}$ is bounded from above on $X$ by $M^{1-\delta}=
\max_X (|s|^2e^{-\ol\varphi_Y})^{1-\delta}<+\infty$ and we can take $C_\delta=C_0M^{1-\delta}\delta^{-2}$. Therefore, \eqref{34} follows from \eqref{32} and \eqref{35}. 
If we choose $\delta\le\varepsilon_0$ and recall that
$\varphi_Y=\varphi_{G_1}-\varphi_{G_2}$, we see that inequality 
(\ref{27}$_\varepsilon)$ implies
\begin{equation}
\label{35bis}
\delta\varphi_{Y,\varepsilon}+(1-\delta)\ol\varphi_Y+\varphi_{F,\varepsilon}\le
\delta\varphi_{G_1,\varepsilon}+(\varepsilon_0-\delta)\varphi_{G_2,\varepsilon}
+(1-\delta)\ol\varphi_Y+C(k).
\end{equation}
For $i\in \{1, 2\}$ the function $\varphi_{G_i,\varepsilon}$ is psh, thus in particular uniformly
bounded from above on $X_k$ by a constant independent of $\varepsilon$.
Hence $\delta\varphi_{Y,\varepsilon}+(1-\delta)\ol\varphi_Y+\varphi_{F,\varepsilon}$ is 
uniformly bounded from above by a constant~$C_3(k)$ if we fix e.g.\ $\delta=\varepsilon_0$, and thanks to \eqref{34}
the {\sl unweighted} norm of $U_{k, \varepsilon}$ admits a bound
$$\int_{X_k}|U_{k, \varepsilon}|^2\leq C_4(k).$$
We stress here the fact that the constant is independent of $\varepsilon$. Therefore we can extract a subsequence that is uniformly convergent on all compact subsets of~$X_k$. 
Indeed, this follows from the classical Montel theorem, which in turn is a consequence of the Cauchy integral formula to prove equicontinuity (this is where we use the fact that $C_4(k)$
is independent of $\varepsilon$), coupled with
Arzel\`a-Ascoli theorem. 
Let $U_k$ be the corresponding limit. Fatou's lemma shows that we have the estimate
\begin{equation}\label{36}\int_{X_k}{|U_{k}|^2e^{-\varphi_{Y}- \varphi_{F}}\over 
|s|^2e^{-\varphi_{Y}}
\log^2(|s|^2e^{-\varphi_{Y}})}\leq C_0\int_{Y\cap(X\setminus H)}
|u|^2e^{-\varphi_F}.\end{equation}
If we now let $k\to +\infty$, we get a convergent subsequence $U_k$
with limit $U=\lim U_k$ on $X\setminus H=\bigcup X_k$, which is uniform on all compact subsets
of $X\setminus H$, and such that
\begin{equation}\label{37}\int_{X\setminus H}{|U|^2e^{-\varphi_{Y}- \varphi_{F}}\over 
|s|^2e^{-\varphi_{Y}}
\log^2(|s|^2e^{-\varphi_{Y}})}\leq C_0\int_{Y\cap(X\setminus H)}
|u|^2e^{-\varphi_F}.\end{equation}
We can reinterpret $U$ as a section of $(K_{X}+Y+ F)|_{X\setminus H}$
satisfying the equality
\begin{equation}\label{38}U|_{Y\cap(X\setminus H)}= u\wedge ds.\end{equation}
Then the estimate \eqref{37} is in fact an intrinsic estimate in terms of the 
hermitian metrics (i.e.\ independent of the specific choice of 
trivialization we have made, especially since $H$ is of measure zero
with respect to the $L^2$ norms). The proof of \eqref{34} also shows that
\begin{equation}\label{39}\int_{X\setminus H}|U|^2e^{-\delta \varphi_Y-(1-\delta)\ol\varphi_Y -
\varphi_F}\leq C_\delta\int_{Y\cap(X\setminus H)}|u|^2e^{-\varphi_F}.\end{equation}
In a neighborhood of any point $x_0\in H$, the weight 
$\delta \varphi_Y+(1-\delta)\ol\varphi_Y+\varphi_F$ expressed with respect to
a local trivialization of $F$ near $x_0$ is locally bounded from above
by (\ref{35bis}), if we take $\delta\le\varepsilon_0$. We conclude that $U$ extends holomorphically to $X$ and Theorem \ref{t-OT} is proved.
\end{proof}

\medskip

\begin{rem}{\rm In the absence of hypothesis \eqref{27}, the uniform bound
arguments in the proof of \eqref{t-OT} collapse. In particular the limit $U$ might
acquire poles along $H$. Fortunately, the exact value of
$\varepsilon_0$ does not matter quantitatively, and the estimates we get 
at the end are independent of the constant $\varepsilon_0$.  }
\end{rem}

\medskip

\noindent

\section{Proof of Theorem \ref{t-ext}}\label{s-pot}

\noindent In the present section, we will prove Theorem \ref{t-ext}. Our proof 
relies heavily on Theorem \ref{t-OT}. However, in order to better understand the relevance of the technical statements which follow (see Theorem \ref{t-5} below), we first consider a particular case of \eqref{t-ext}.

\begin{theorem}\label{t-4} Let $\{ S, Y_j\}$ be a set of hypersurfaces of 
a smooth, projective manifold $X$
having normal crossings. Assume also that there exists rational numbers
$0< b^j< 1$ such that $K_X+S+B$ is hermitian semi-positive, where 
$B= \sum_jb^jY_j$ and that 
there exists an effective $\bQ$-divisor $D:= \sum_j \nu^jW_j$ on $X$, numerically equivalent to $K_X+S+B$, such that $S\subset {\rm Supp}(D)\subset {\rm Supp} (S+B)$.
Let $m_0$ be a positive integer, such that $m_0(K_X+S+B)$ is Cartier. Then every section $u$ of the line bundle $\mathcal O _S(m_0(K_S+ B|_{S}))$ 
extends to~$X$.
\end{theorem}

\begin{proof} Let $h_0= e^{-\varphi_0}$ be a smooth metric on $K_X+S+B$, with semi-positive curvature. As in the introductory Subsection \ref{s-mmps}, we have 
$$\mathcal O_X(N(K_X+S+B))\cong \mathcal O_X(ND)\otimes \rho^N$$
where $N$ is a sufficiently divisible positive integer, and $\rho$ is a 
line bundle on $X$, which admits a metric $h_\rho= e^{-\varphi_\rho}$
whose curvature form is equal to zero.
\smallskip

\noindent We can assume that 
\begin{equation}\label{99}\varphi_0\geq \varphi_\rho+ \sum_j\nu^j\log|f_{W_j}|^2\end{equation}
i.e., that $h_0$ is less singular that the metric induced 
by the divisor $\sum_j \nu^jW_j$, simply by adding to the local weights of $h_0$ a sufficiently large constant. 

Assume that $S= W_1$. We define a metric $\varphi_S$ on the line bundle corresponding to $S$ such that the following equality holds
$$\varphi_0= \nu^1\varphi_S+ \varphi_\rho+ \sum_{j\neq 1}\nu^j\log|f_{W_j}|^2.$$
In order to apply Theorem \ref{t-OT}, we write
$$m_0(K_X+S+B)= K_X+S+B+ (m_0-1)(K_X+S+B)$$
and we endow the line bundle corresponding to $F:= B+ (m_0-1)(K_X+S+B)$ with the metric
$\varphi_F:= \varphi_B+ (m_0-1)\varphi_0$. 
\medskip

\noindent Then, we see that the hypothesis of \eqref{t-OT} are verified, as follows.
\smallskip

\noindent $\bullet$ We have $|f_S|^2e^{-\varphi_S}\leq 1$ by the inequality 
\eqref{99} above.
\smallskip

\noindent $\bullet$ We have $\Theta_{h_F}(F)\geq 0$, as well as 
$\displaystyle \Theta_{h_F}(F)\geq {1\over \alpha}\Theta_{h_S}\big(\mathcal O_X(S)\big)$, for any 
$\displaystyle \alpha\geq {1\over (m_0-1)\nu^1}$.
\smallskip

\noindent In order to apply Theorem \eqref{t-OT}, we define 
$$\alpha:= \max\Big\{ 1, {1\over (m_0-1)\nu^1}\Big\}$$
and we rescale the metric $h_S$ by a constant, as follows
$$\varphi_S^\alpha:= \varphi_S+ \alpha.$$
Then we have $\displaystyle |f_S|^2e^{-\varphi_S^\alpha}\leq e^{-\alpha}$
thanks to the first bullet above, and moreover the curvature conditions
$$\Theta_{h_F}(F)\geq 0, \quad \Theta_{h_F}(F)\geq {1\over \alpha}\Theta_{h_S^\alpha}\big(\mathcal O_X(S)\big)$$
are satisfied.
\smallskip

\noindent $\bullet$ The (rescaled) weight $\varphi_S^\alpha$ can be written as the difference of two
psh functions, $\varphi _{G_1}-\varphi _{G_2}$ where $\varphi _{G_1}:=\alpha+ \varepsilon _0 \varphi _0$ and  $\varphi _{G_2}:=\varepsilon _0 \big(\varphi_\rho+ \sum_{j\neq 1}\nu^j\log|f_{W_j}|^2\big)$ and $\varepsilon _0=1/\nu ^1$.
We have
$$C+ \varepsilon_0\big(\varphi_\rho+ \sum_{j\neq 1}\nu^j\log|f_{W_j}|^2\big)\geq \varphi_F$$
by the assumption concerning the support of $D$.

\medskip
\noindent We also have $\displaystyle \int_S|u|^2e^{-\varphi_F}< \infty$,
since $(X, B)$ is klt and $h_0$ is non-singular. Therefore, by Theorem \ref{t-OT}
the section $u$ extends to $X$.
\end{proof}

\begin{rem} The norm of the extension we construct by this procedure will depend only on $C_\delta$ computed in Theorem \ref{t-OT}, the rescaling factor
$\displaystyle e^{\delta\alpha}$ where $\displaystyle \alpha:=\max \Big\{ 1, {1\over (m_0-1)\nu^1} \Big\}$ and $\displaystyle \int_S|u|^2e^{-\varphi_F}$.
\end{rem}

\subsection{Construction of potentials for adjoint bundles}
\medskip

\noindent As one can see, the hypothesis (3) of \eqref{t-ext} is much weaker than the corresponding one in \eqref{t-4} (i.e. the hermitian semi-positivity of
$K_X+S+B$), and this induces many complications 
in the proof. The aim of Theorem \ref{t-5} below is to construct a substitute 
for the smooth metric $h_0$, and 
it is the main technical tool in the proof of \eqref{t-ext}.

\medskip
 
\begin{theorem}\label{t-5} Let $\{ S, Y_j\}$ be a smooth hypersurfaces of $X$
with normal crossings. Let $0< b^j< 1$ be rational numbers, such that:

\begin{enumerate}

\item[\rm (1)] We have $K_X+ S+ \sum_jb^jY_j\equiv \sum_j\nu^jW_j$, where $\nu^j$ are positive rational numbers,
and  $\{ W_j\}\subset \{S, Y_j\}$.

\smallskip
\item[\rm (2)] Let $m_0$ be a positive integer such that $m_0(K_X+ S+ \sum_jb^jY_j)$ is Cartier, and there exists a  non-identically zero section $u$ of 
$\displaystyle \mathcal O _S(m_0(K_S+ \sum_jb^jY_j|_S))$.

\smallskip
\item[\rm (3)] Let $h$ be a non-singular, fixed metric on the $\Q$-line bundle $K_X+ S+ \sum_jb^jY_j$; then there exists a sequence
$\{ \tau_m\}_{m\geq 1}\subset L^1(X)$, such that $\displaystyle \Theta_h(K_X+ S+ \sum_jb^jY_j)+ \sqrt{-1}\ddbar \tau_m\geq
-{1\over m}\omega$ as currents on $X$, the restriction ${\tau }_m|_S$ is well defined and we have
\begin{equation}\label{48}\tau_m|_S\geq C(m)+ \log|u|^{2\over m_0}\end{equation}
where $C(m)$ is a constant, which is allowed
to depend on $m$.

\end{enumerate}
\smallskip

\noindent  Then there exists a constant $C< 0$ independent of $m$, and a sequence of functions $\{f_m\}_{m\geq 1}\subset L^1(X)$ such that:

\begin{enumerate}

\smallskip

\item[\rm (i)] We have $\sup_X f_m= 0$, and moreover $\displaystyle \Theta_h(K_X+ S+ \sum_jb^jY_j)+ \sqrt{-1}\ddbar f_m\geq
-{1\over m}\omega$ as currents on $X$.

\smallskip
\item[\rm (ii)] The restriction $f_m|_S$ is well-defined, and we have 
\begin{equation}\label{49}f_m|_S\geq C+ \log|u|^{2\over m_0}.\end{equation}
 
 \end{enumerate}

\end{theorem}

\medskip
\noindent 
The proof of Theorem \ref{t-5} follows an iteration scheme, that we now explain. We start
with the potentials $\{ \tau_m\big\}$ provided by the hypothesis (3) above; then we construct potentials
$\{ \tau_m^{(1)}\}$ such that the following properties are satisfied.

\begin{enumerate}

\item[\rm (a)] We have $\sup_X \tau_m^{(1)}= 0$, and moreover 
$$\displaystyle \Theta_h(K_X+ S+ \sum_jb^jY_j)+ \sqrt{-1}\ddbar \tau_m^{(1)}\geq
-{1\over m}\omega$$ in the sense of currents on $X$.

\smallskip
\item[\rm (b)] The restriction $\tau_m^{(1)}|_S$ is well-defined, and
 there exists a constant $C$ independent of $m$ such that
$$\tau_m^{(1)}\geq C+ \log|u|^{2\over m_0}+ \rho\sup_S\tau_m,$$ 
at each point of $S$ (where $0<\rho< 1$ is to be determined).
\end{enumerate}
The construction of
$\{\tau_m^{(1)}\}$ with the pertinent curvature and uniformity properties (a) and (b) is possible by
Theorem \ref{t-OT}. Then we repeat this procedure: starting with $\{\tau_m^{(1)}\}$
we construct $\{\tau_m^{(2)}\}$, and so on. Thanks to the uniform estimates we provide during this process, the limit of $\{\tau_m^{(p)}\}$ as $p\to \infty$
will satisfy the requirements of \eqref{t-5}. We now present the details.
\smallskip

\begin{proof}
Let $B:= \sum_jb^jY_j$; by hypothesis, there exists $\{ \tau_m\}\subset  L^1(X)$, such that
\begin{equation}\label{50}\max_X\tau_m= 0, \quad \Theta_h(K_X+ S+ B)+ \sqrt{-1}\ddbar \tau_m\geq
-{1\over m}\omega\end{equation}
on $X$.
We denote by $D=\sum \nu ^jW_j$ the $\bQ$-divisor provided by hypothesis (1) of \eqref{t-5}; and let
$\displaystyle \tau_D:=\log |D|_{h\otimes h_\rho^{-1}}^{2}$
(by this we mean the norm of the $\mathbb Q$-section associated to $D$, measured
with respect to the metric $h\otimes h_\rho^{-1}$,
 cf. the proof of \eqref{t-4}) be the logarithm of its norm. We can certainly assume that $\tau_D\leq 0$. By replacing $\tau_m$
by $\max \{\tau_m, \tau_D\}$, the relations \eqref{50} above are still satisfied 
(cf. \S \ref{sing_metrics}) and in addition we can assume that we have
\begin{equation}\label{51}\tau_m\geq \tau_D\end{equation}
at each point of $X$.
\smallskip

Two things can happen: either $S$ belongs to the set $\{ W_j\}$, or not. In the later case there is nothing to prove as the restriction $\tau_m |_S$ is well defined, so we assume that $S= W_1$.
After the normalization indicated above,
we define a
metric $\displaystyle e^{-\psi_{S, m}}$ on $\cO_X(S)$ which will be needed in order to apply \eqref{t-OT}.
Let
$$\displaystyle \varphi_{\tau_m}:= \varphi_h+ \tau_m$$
be the local weight of the metric $e^{-\tau_m}h$ on $K_X+S+B$.
The metric $\psi_{S, m}$ is defined so that the following {\sl equality} holds
\begin{equation}\label{52}\varphi_{\tau_m}= \nu^1\psi_{S, m}+ \varphi_\rho+ \sum_{j\neq 1}\nu^j\log |f_{W_j}|^2\end{equation}
(cf. Lemma \ref{l-weight})
where $\displaystyle f_{W_j}$ is a local equation for the hypersurface $W_j$. We use here the hypothesis (1) of Theorem \ref{t-5}.
Then the functions $\psi_{S, m}$ given by equality \eqref{52} above are the
local weights of a metric on $\cO_X(S)$.

\medskip

\noindent The norm/curvature properties of the objects constructed so far are listed below.

\begin{enumerate}

\item [{\rm (a)}] {\sl The inequality $|f_S|^2e^{-\psi_{S, m}}\leq 1$ holds at each point of $X$.} Indeed, this is a direct consequence of the relations \eqref{51} and \eqref{52} above (cf. \eqref{l-weight}).
\smallskip

\item [{\rm (b)}] {\sl We have $\displaystyle \Theta_{\varphi_{\tau_m}}(K_X+S+B)\geq
-{1\over m}\omega$.}
\smallskip

\item [{\rm (c)}] {\sl We have $\displaystyle \Theta_{\varphi_{\tau_m}}(K_X+S+B)\geq
\nu^1\Theta_{\psi_{S, m}}(S)$.} This inequality is obtained as a direct consequence of \eqref{52}, since the curvature of $h_\rho$ is equal to zero.

\smallskip

\item [{\rm (d)}] {\sl For each $m$ there exists a constant $C(m)$ such that
$$\tau_m|_S\geq C(m) + \log |u|^{2\over m_0}.$$}
Indeed, for this inequality we use the hypothesis (3) of Theorem \ref{t-5},
together with the remark that the renormalization and the maximum we have used to insure \eqref{51} preserve this hypothesis.

\end{enumerate}

\medskip
\noindent As already hinted, we will modify each element of the sequence of functions $\{\tau_m\}_{m\geq 1}$ by using some ``estimable extensions" of the section $u$ (given by hypothesis 2) and its tensor powers, multiplied by a finite 
number of auxiliary sections of some ample line bundle. 
Actually, we will concentrate our efforts on one single index e.g. $m= km_0$, and try to understand the uniformity properties of the  constants involved in the computations.

In order to simplify the notations, let $\tau:= \tau_{km_0}$ and denote by $\psi_S$ the
metric on $\cO_X(S)$ defined by the equality
\begin{equation}\label{53}\varphi_{\tau}= \nu^1\psi_{S}+ \varphi_\rho+ \sum_{j\neq 1}\nu^j\log |f_{W_j}|^2.\end{equation}
Even if this notation does not make it explicit, we stress here the fact that the metric
$\psi_S$ depends on the function $\tau$ we want to modify.

We consider a non-singular metric $h_S= e^{-\varphi_S}$ which is independent of $\tau$, and for each $0\leq \delta\leq 1$, we define the convex combination metric
\begin{equation}\label{54}\psi_{S}^\delta:= \delta\psi_{S}+ (1- \delta)\varphi_S.\end{equation}
The parameter $\delta$ will be fixed at the end, once we collect all the requirements
we need it to satisfy.

We assume that the divisor $A$ is sufficiently ample, so that the metric $\omega$ in
(b) above is the curvature of the metric $h_A$ on $\mathcal O _X(A)$ induced by its global sections say $\{ s_{A, i}\}$.

Now we consider the section $u^{\otimes k}\otimes s_A$ of the line bundle
$$\mathcal O _S(km_0(K_S+B|_S)+ A|_S),$$where $s_A\in \{ s_{A, i}\}$, and we define the
set
$${\cE}:= \big\{U\in H^0\big(X, \mathcal O _X(km_0(K_X+S+B)+ A)\big) : U|_S= u^{\otimes k}\otimes s_A \big\}.$$

\noindent A first step towards the proof of Theorem \ref{t-5} is the following statement
(see e.g. \cite{BP10}).

\begin{lemma} The set $\cE$ is non-empty; moreover, there exists an element
$U\in \cE$ such that the following integral is convergent
\begin{equation}\label{55}\Vert U\Vert^{{2\over km_0}(1+\delta)}:= \int_X|U|^{{2\over km_0}(1+ \delta)}e^{-\delta\varphi_\tau-
\psi_{S}^\delta- \varphi_B}e^{-{1+ \delta\over km_0}\varphi_A}< \infty\end{equation}
as soon as $\delta$ is sufficiently small.
\end{lemma}

\begin{proof}
We write the divisor $km_0(K_X+S+B)+ A$ in adjoint form
as follows
$$km_0(K_X+S+B)+ A= K_X+ S+ F$$
where
$$F:=  B+ (km_0-1)(K_X+S+B)+ A.$$
We endow the line bundle $\mathcal O _X(F)$ with the metric whose local weights are
\begin{equation}\label{56}\varphi_F:= \varphi_B+ (km_0-1)\varphi_\tau+ \varphi_A.\end{equation}
By property (b), the curvature of this metric is greater than $\displaystyle
{1\over km_0}\omega$, and the section $u^{\otimes k}\otimes s_A$
is integrable with respect to it, by property (d).

The classical Ohsawa-Takegoshi theorem shows the existence of a section
$U$ corresponding to the divisor $km_0(K_X+S+B)+ A$, such that
$\displaystyle U|_S= u^{\otimes k}\otimes s_A$, and such that the following integral is convergent
\begin{equation}\label{57}\int_X|U|^2e^{-\varphi_B- (km_0-1)\varphi_\tau- \varphi_A-\varphi_S}< \infty .\end{equation}
By the H\"older inequality we obtain
\begin{equation}\label{58}\int_X|U|^{{2\over km_0}(1+ \delta)}e^{-\delta \varphi_\tau-
\psi_{S}^{\delta}- \varphi_B}e^{-{1\over km_0}(1+ \delta)\varphi_A}\leq
C I^{km_0-1-\delta\over km_0}\end{equation}
where we denote by $I$ the following quantity
\begin{equation}\label{59}I:= \int_Xe^{\varphi_\tau-{km_0\over km_0-1-\delta}\psi_{S}^{\delta}-\varphi_B},\end{equation}
and $C$ corresponds to the integral \eqref{57} raised to the power $\displaystyle 
{1+\delta\over km_0}$.
In the above expression we have skipped a non-singular metric corresponding to $\varphi_S$.
By relation \eqref{53}, the integral $I$ will be convergent, provided that
\begin{equation}\label{60}{\delta\over \nu^1}\leq {1\over 2}\end{equation}
so that the lemma is proved.
\end{proof}
\medskip

\noindent We consider next an element $U\in {\mathcal E}$ for which the semi-norm
\eqref{55} is {\sl minimal}. Let $\{ U_p\}_{p\geq 1}\subset  {\cE}$ such that the sequence
$$n_p:= \int_X|U_p|^{{2\over km_0}(1+ \delta)}e^{-\delta\varphi_\tau-
\psi_{S}^\delta- \varphi_B}e^{-{1\over km_0}(1+ \delta)\varphi_A}$$
converges towards the infimum say $n_\infty$ of the quantities \eqref{55}  when $U\in {\cE}$. From this, 
we infer that $\{ U_p\}_{p\geq 0}$ has a convergent subsequence, and obtain our minimizing section as its limit. This can be justified either by H\"older inequality,
or by observing that we have
$$\psi_{S}^\delta+ \varphi_B\leq C+ {\delta\over \nu^1}\varphi_\tau
-\delta\sum_{j\neq 1}{\nu^j\over \nu^1}\log |f_{Y_j}|^2+ \sum_{j\neq 1}b^j
\log|f_{Y_j}|^2,$$
where the last term is bounded from above, as soon as $\delta$ verifies the inequalities
\begin{equation}\label{61}\delta{\nu^j\over \nu^1}\leq b^j\end{equation}
for all $j$.

\smallskip

In conclusion, some element $U_{\rm min}^{(km_0)}\in {\cE}$ with minimal
semi-norm exists, by the usual properties of holomorphic functions,
combined with the Fatou Lemma. Next, we will show that the semi-norm of
$U_{\rm min}^{(km_0)}$ is bounded in a very precise way
(again, see \cite{BP10} for similar ideas).
\smallskip

\noindent To this end, we first construct an extension $V$ of $u^{\otimes k}\otimes s_A$ by using $U_{\rm min}^{(km_0)}$ \emph{as a metric}; this will be done by using \eqref{t-OT}, so
we first write
$$km_0(K_X+S+B)+ A= K_X+ S+ F$$
where
$$F:=  B+ (km_0-1-\delta)(K_X+S+B+ {1\over km_0}A)+ \delta(K_X+S+B)+ {1+\delta\over km_0}A.$$
We endow the line bundle $\mathcal O _X(F)$ with the metric whose local weights are
\begin{equation}\label{62}\varphi_F:= \varphi_B+ {km_0-1-\delta\over km_0}\log|U_{\rm min}^{(km_0)}|^2
+ \delta \varphi_\tau+ {1+\delta\over km_0}\varphi_A\end{equation}
and the line bundle corresponding to $S$ with the metric $\displaystyle e^{-\psi_{S}}$
previously defined in \eqref{53}. Prior to applying \eqref{t-OT} we check here the hypothesis \eqref{24}-\eqref{27}. Thanks to \eqref{51} and \eqref{53} we have
\begin{equation}\label{63}|s|^2e^{-\psi_{S}}\leq 1.\end{equation}
Moreover, we see (cf. \eqref{53}) that  
we have 
$$\psi_S= {1\over \nu^1}
\big(\varphi_\tau- \varphi_\rho- \sum_{j\neq 1}\nu^j\log |f_{W_j}|^2\big)$$
hence the requirement \eqref{27} amounts to showing that we have
\begin{equation}\label{64}C_1(k, \delta)+ \sum_{j\neq 1}{\nu^j\over \nu^1}\log |f_{W_j}|^2\geq C_2(k, \delta)\varphi_F\end{equation}
where $C_j(k, \delta)$ are sufficiently large  constants, depending (eventually) on the norm of the section $U_{\rm min}^{(km_0)}$, and on $\delta$. The existence of such quantities is clear, because of the hypothesis (1) of \ref{t-5} and of the presence of 
$\varphi_B$ in the expression \eqref{62}.

 The curvature hypothesis required by \eqref{t-OT} are also satisfied, since we have
\begin{equation}\label{65}\Theta_{h_F}(F)\geq 0, \quad \Theta_{h_F}(F)\geq {\delta \nu_1}
\Theta_{h_{S}}(S)\end{equation}
by relations \eqref{62} and \eqref{53} (the slightly negative part of the Hessian of $\varphi_\tau$ is compensated by the Hessian of $\displaystyle {1\over km_0}\varphi_A$).
\smallskip

Therefore, we are in position to apply Theorem \ref{t-OT} and infer the existence of an element $V_\delta\in \cE$ such that
\begin{equation}\label{66}\int_X{|V_\delta|^2\over |U_{\rm min}^{(km_0)}|^{2{km_0-1-\delta\over km_0}}}
e^{-\delta\varphi_\tau- {1+\delta\over km_0}\varphi_A-\varphi_B-\psi_{S}^\delta}\leq C(\delta)\int_S|u|^{1+\delta\over m_0}e^{-\varphi_B- \delta\varphi_\tau}.\end{equation}

\noindent The constant $C(\delta)$ in \eqref{66} above is obtained by Theorem
\ref{t-OT}, via the rescaling procedure described in the proof of \eqref{t-4} adapted to the current context; we stress on the fact that this constant in completely independent 
of $k$ and $\varphi_\tau$.
\smallskip

We will show next that it is possible to choose $\delta:= \delta_0$ small enough,
independent of $k$ and $\tau$, such that the right hand side of \eqref{66} is smaller than
$\displaystyle Ce^{-\delta_0\sup_S(\tau)}$. Here and in what follows we will freely 
interchange $\tau$ and $\varphi_\tau$, as they differ by a function which only depends on the fixed metric $h$ on $K_X+S+B$; in particular, the difference $\tau- \varphi_\tau$ equals a quantity independent of the family of potentials we are trying to construct.

\noindent Prior to this, we recall the following basic result,
originally due to L.~H\"ormander for open sets in $\bC^n$, and to G. Tian in
the following form.

\begin{lemma}\label{5.5} \cite{Tian} Let $M$ be a compact complex manifold, and let $\alpha$ be
  a real, closed $(1,1)$-form on $M$. We consider the family of normalized potentials
$$\cP:=  \big\{f\in L^1(X) : \sup_M(f)= 0, \quad \alpha+ \sqrt{-1}\ddbar f\geq 0\big\}.$$
Then there exists constants $\gamma_H> 0$ and $C_H> 0$ such that
\begin{equation}\label{67}\int_Me^{-\gamma_H f}dV\leq C_H\end{equation}
for any $f\in \cP$. In addition, the numbers $\gamma_H$ and  $C_H$ are uniform with respect to
$\alpha$.
\end{lemma}

\medskip

\noindent We will use the previous lemma as follows: first we notice that we have
$$\int_S|u|^{1+\delta\over m_0}e^{-\varphi_B- \delta\varphi_\tau}= 
\int_S|u|^{1+\delta\over m_0}e^{-\varphi_B- \delta\varphi_h- \delta\tau}$$
simply because by definition the equality $\varphi_\tau= \varphi_h+ \tau$
holds true on $X$. 

The pair $(S, B|_S)$ is klt, hence there exists a positive real number
$\mu_0> 0$ such that $\displaystyle e^{-(1+ \mu_0)\varphi_B}\in L^1_{\rm loc}(S)$.
By the H\"older inequality, we have
$$\int_S|u|^{1+\delta\over m_0}e^{-\varphi_B- \delta\varphi_h- \delta\tau}\leq
C\Big(\int_S e^{-{(1+\mu_0)\over \mu_0}\delta\tau}dV\Big)^{\mu_0\over 1+ \mu_0}
$$
where $dV$ is a non-singular volume element on $S$, and the constant $C$
above is independent of $\tau$ and of $\delta$, provided that $\delta$ belongs to a fixed compact set (which is the case here, since $0\leq \delta\leq 1$).

Next, in order to obtain an upper bound of the right hand side term of the preceding inequality, by the above lemma, we can write 
$$\Big(\int_S e^{-{(1+\mu_0)\over \mu_0}\delta\tau}dV\Big)^{\mu_0\over 1+ \mu_0}
= e^{-\delta\sup_S(\tau)}
\Big(\int_S e^{-{(1+\mu_0)\over \mu_0}\delta(\tau- \sup_S(\tau))}dV\Big)^{\mu_0\over 1+ \mu_0}.$$

\noindent We fix now $\delta:= \delta_0$ small enough such that the following conditions are satisfied:

\noindent $\bullet$ We have $\displaystyle {\delta_0\over \nu^1}< {1\over 2}$ and $\delta_0\nu^j\leq \nu^1b^j$
for all $j\neq 1$ (we recall that $S= Y_1$).

\noindent $\bullet$ The inequality below holds
$$\displaystyle \delta_0{1+\mu_0\over \mu_0}\leq \gamma_H,$$ 
where $\gamma_H$ is given by Lemma \ref{5.5} applied to the next data: $M= S$
and $\displaystyle \alpha:= \Theta_h(K_S+ B)+ {1\over km_0}\omega$ (where $h$ here is the restriction of the metric in (3)
of \eqref{t-5} to $S$). We notice that $\gamma_T, C_T$ can be assumed to be uniform
with respect to $k$, precisely because of the uniformity property mentioned at the end of Lemma \ref{5.5}.

\smallskip

\noindent The conditions we imposed on $\delta_0$ in the two bullets above
are independent of the particular potential $\tau$ we choose. Hence,
in the  proof we first fix $\delta_0$ as above, then construct
the minimal element $U_{\rm min}^{(km_0)}$ and after that,
we produce an element $V\in \cE$ such that
\begin{equation}\label{68}I(V):= \int_X{|V|^2\over |U_{\rm min}^{(km_0)}|^{2{km_0-1-\delta_0\over km_0}}}
e^{-\delta_0\varphi_\tau- {1+\delta_0\over km_0}\varphi_A-\varphi_B-\psi_{S}^{\delta_0}}\leq Ce^{-\delta_0\sup_S(\tau)},\end{equation}
as a consequence of \eqref{66}, Lemma \ref{5.5} and the above explanations.

\smallskip

\noindent On the other hand, we claim that we have
\begin{equation}\label{69}\int_X|U_{\rm min}^{(km_0)}|^{{2\over km_0}(1+ \delta_0)}e^{-\delta_0\varphi_\tau-
\psi_{S}^{\delta_0}- \varphi_B}e^{-{1\over km_0}(1+ \delta_0)\varphi_A}\leq
I(V).\end{equation}
Indeed, this is a consequence of H\"older inequality: if relation \eqref{69} fails to hold, then it is easy to show that the quantity
$$\int_X|V|^{{2\over km_0}(1+ \delta_0)}e^{-\delta_0\varphi_\tau-
\psi_{S}^{\delta_0}- \varphi_B}e^{-{1\over km_0}(1+ \delta_0)\varphi_A}$$
is {\sl strictly smaller} than the left hand side of \eqref{69}--and this contradicts the
the minimality
property of the $U_{\rm min}^{(km_0)}$.

\medskip
\noindent In conclusion, we have the inequality
\begin{equation}\label{70}\int_X|U_{\rm min}^{(km_0)}|^{{2\over km_0}(1+ \delta_0)}e^{-\delta_0\varphi_\tau-
\psi_{S}^{\delta_0}- \varphi_B}e^{-{1\over km_0}(1+ \delta_0)\varphi_A}\leq
Ce^{-\delta_0\sup_S(\tau)},\end{equation}
by combining \eqref{68} and \eqref{69}.

Let $\{ s_{A, i}\}_{i=1,\ldots ,M}$ be the finite set of sections of $A$ which were fixed at the beginning of the proof. Then we can construct an extension
$U_{\rm min, i}^{(km_0)}$ of $u^{\otimes k}\otimes s_{A, i}$, with bounded
$L^{2\over km_0}$ norm  as in \eqref{70}. We define the function
\begin{equation}\label{71}\wt \tau:= {1\over km_0}\log\sum_i|U_{\rm min, i}^{(km_0)}|_{h^{\otimes km_0}\otimes h_A}^2\end{equation}
and we observe that we have
\begin{equation}\label{72}\int_Xe^{(1+\delta_0)\wt \tau} dV\leq Ce^{-\delta_0\sup_S(\tau)},\end{equation}
as a consequence of \eqref{70}, since the function $\tau$ (or if one prefers, $\varphi_\tau$) is  \emph{negative}, and the part of $\psi_{S}^{\delta_0}$
having the ``wrong sign" is absorbed by $\varphi_B$.

Next, as a consequence of the mean inequality for psh functions, together with the
compactness of $X$ (see Lemma \eqref{l-C}) we infer from \eqref{72} that
\begin{equation}\label{73}\wt \tau(x)\leq C- {\delta_0\over 1+\delta_0}\sup_S(\tau)\end{equation}
for any $x\in X$. Again, the constant ``$C$" has changed since \eqref{72}, but in a manner
which is universal, i.e. independent of $\tau$.

We also remark that the restriction to $S$ of the
functions $\displaystyle \wt \tau$ we have
constructed is completely determined by
$$\wt \tau |_{S}= \log|u|^{2\over m_0}_h.$$
In order to re-start the same procedure, we introduce the functions
\begin{equation}\label{74}\tau^{(1)}:= \wt \tau- \sup_X(\wt \tau)\end{equation}
and we collect their properties below:

\item {\rm (N)} $\sup_X\tau^{(1)}= 0$;
\smallskip

\item{\rm (H)} $\displaystyle \Theta_h(K_X+S+B)+ \sqrt{-1}\ddbar \tau^{(1)}\geq -{1\over km_0}\omega$ as currents on $X$;
\smallskip

\item{\rm (R)} $\displaystyle \tau ^{(1)}|_{S}\geq {\delta_0\over 1+\delta_0}
\sup_S\tau- C+ {1\over m_0}\log|u|^2_h$; in particular, we have 
\begin{equation}\label{75}\sup_S\tau^{(1)}\geq {\delta_0\over 1+\delta_0}
\sup_S\tau- C .\end{equation}
In order to see that the last inequality holds, we fix {\sl any} point $x_0\in S$, such that $u$ is does not vanish at this point.
After possibly rescaling $u$, we can assume that $\displaystyle {1\over m_0}\log|u({x_0})|^2_h= 0$
and then we have 
$$ \tau ^{(1)}(x_0)\geq {\delta_0\over 1+\delta_0}
\sup_S\tau- C+ {1\over m_0}\log|u({x_0})|^2_h.$$
We certainly have $\sup_S\tau^{(1)}\geq \tau ^{(1)}(x_0)$; hence \eqref{75}
follows. 
 
\medskip

\noindent We remark next that the restriction properties of $\tau^{(1)}$ have improved: modulo the function 
$$x\to \displaystyle -C+ {1\over m_0}\log|u(x)|^2_h$$ 
which is independent of
$\tau$, we ``gain" a factor $\displaystyle {\delta_0\over 1+ \delta_0}< 1$. 

\begin{rem} It is of paramount importance to realize that the constant 
``C" appearing in the expression of the function above is {\sl universal}: it only depends on the geometry of $(X, S)$, the metric $h$ and the norm of $u$, also on the form
which compensate the negativity of the Hessian of $\tau$ (cf. \eqref{66}, Tian's lemma
\ref{5.5} and the comments after that). Moreover, if these quantities are varying in a uniform manner (as it is the case for our initial sequence $\{\tau_m\}$), then the constant $C$ can be assumed to be independent of $m$.
\end{rem}

Therefore, it is natural to repeat this procedure with the sequence of functions
$\displaystyle \tau^{(1)}$ as input, so we successively produce the potentials
$\displaystyle \{\tau^{(p)}\big\}$ with the properties (N), (H) and (R) as above,
for each $p\geq 1$. We stress again on the fact that the quantities $C, \delta_0$
have two crucial
uniformity properties: with respect to $p$ (the number of iterations) and to $m$
(the index of the sequence in \ref{t-5}).

\smallskip

\noindent Proceeding by induction we show that the following two equations hold:
\begin{equation}\label{76}\tau ^{(p)}|_{S}\geq \Big({\delta_0\over 1+\delta_0}\Big)^p
\sup_S\tau- C\delta_0\Big(1-  \Big({\delta_0\over 1+\delta_0}\Big)^{p-1}\Big)- C+ \log|u|_h^{2\over m_0},\end{equation}
\begin{equation}\label{76b}\sup_S\tau^{(p)}\geq \Big({\delta_0\over 1+\delta_0}\Big)^p
\sup_S\tau- C(1+\delta_0)\Big(1-  \Big({\delta_0\over 1+\delta_0}\Big)^{p}\Big).\end{equation}
The relation \eqref{76b} is obtained by successive applications of \eqref{75} of the relation (R), since we have
$$\sup_S\tau^{(p)}\geq {\delta_0\over 1+\delta_0}
\sup_S\tau^{(p-1)}- C$$
for any $p\geq 1$. The lower bound \eqref{76} is derived as a direct consequence of 
\eqref{76b} and \eqref{75}.

\medskip
\noindent The functions required by Theorem \ref{t-5} are obtained by standard  facts of pluripotential theory (see e.g.\ \cite{Lelong69}, \cite{Lelong71}, \cite{Klimek}): some subsequence 
$\{\tau_{km_0}^{(p_\nu)}\}$ of $\{\tau_{km_0}^{(p)}\}$
will converge in $L^1$ to the potential $f_{km_0}$, as upper regularized limits
$$f_{km_0}(z):= \lim\sup_{x\to z}\lim_{\nu\to\infty}\tau_{km_0}^{(p_\nu)}(x).$$
By letting $p\to \infty$, we see that the {\sl non-effective} part of the estimate \eqref{76} tends to zero, so Theorem \ref{t-5}
is proved.

\subsection{Proof of the Extension Theorem}

\noindent In this last part of the proof of \eqref{t-ext}, we first recall that by the techniques originating in Siu's seminal article \cite{Siu00}, for any fixed section $s_A\in H^0(X,\mathcal O _X(A))$, the section $u^{\otimes k}\otimes s_A$ extends to $X$, for any $k\geq 1$, (cf. \cite[6.3]{HM10}).
Indeed, by the assumption we have
$$Z_{\pi^\star(u)}+ m_0\wt E|_{\wt S}\geq m_0\Xi$$
(cf. notations in the introduction) and then the extension of $u^{\otimes k}\otimes s_A$ follows.

These extensions provide us with the potentials $\{ \tau_m\}_{m\geq 1}$
(simply by letting $\tau_m=\frac 1 m {\rm log}|U_k|^2$ where $U_k$ is the given extension of $u^{\otimes k}\otimes s_A$)
and then Theorem \ref{t-5} converts this into a quantitative statement: there exists another family of potentials
$\{ f_m\} _{m\geq 1}$ such that
\begin{equation}\label{77}\max_Xf_m= 0, \quad \Theta_h(K_X+S+B)+ \sqrt{-1}\ddbar f_m\geq -{1\over m}\omega\end{equation}
together with
$$f_m|_S\geq C+ \log |u|^{2\over m_0}$$
where $C$ is a constant independent of $m$. Under these circumstances, we invoke the
same arguments as at the end of the preceding paragraph to infer that 
some subsequence 
$\displaystyle \{f_{m_\nu}\}$ of $\{f_m\}$ will converge in $L^1$ to the potential $f_\infty$, as an upper regularized limit
$$f_\infty(z)= \lim\sup_{x\to z}\lim_{\nu\to\infty}f_{m_\nu}(x)$$
for every $z\in X$.

The properties of the limit $f_\infty$ are listed below:
\begin{equation}\label{78}f_{\infty}|_S\geq C+ \log |u|^{2\over m_0}, \quad
\Theta_h(K_X+S+B)+ \sqrt{-1}\ddbar f_\infty\geq 0.\end{equation}
We remark at this point that the metric $e^{-f_\infty}h$ constructed here
plays in the proof of Theorem \ref{t-5} the same role as the metric $h_0$ 
in the arguments we have provided for \eqref{t-4}.

\noindent The rest of the proof is routine: we write
$$m_0(K_X+S+B)= K_X+S+F$$
where we use the following notation
\begin{equation}\label{79}F:= B+ (m_0-1)(K_X+ S+B).\end{equation}
We endow the line bundles $\cO _X(F)$ and $\cO _X(S)$ respectively with the metrics
$$\varphi_{F}:= \varphi_B+ (m_0-1)\varphi_{\infty}$$
and
$$\varphi_{\infty}= \nu^1\varphi_{S}+ \varphi_\rho+ \sum_{j\neq 1}\nu^j\log|f_{W_j}|^2$$
where $h_{\infty}= e^{-\varphi_\infty}$ is the metric given by $e^{-\wt f_\infty}h$;
here we denote 
$$\wt f_\infty:= \max(f_\infty, \tau_D)$$
so that
(as usual) we assume that
\begin{equation}\label{80}\varphi_{\infty}\geq \varphi_\rho+ \sum_j\nu^j\log|f_{W_j}|^2\end{equation}
and $S= Y_1$.

\medskip
\noindent Then the requirements of \eqref{t-OT} are easily checked,  as follows:

\smallskip

\noindent $\bullet$ We have $|s|^2e^{-\varphi_{S}}\leq 1$ by relation \eqref{80} above, and moreover we have the equality
$$\varphi_S= {1\over \nu_1}\big(\varphi_\infty- \varphi_\rho-\sum_{j\neq 1}\nu^j\log|f_{W_j}|^2\big)$$
from which one can determine the hermitian bundles $(G_i, e^{-\varphi_{G_i}})$.
\smallskip

\noindent $\bullet$ There exists $\varepsilon_0> 0$ and $C$ such that
$$\varphi_F\leq \varepsilon_0(\sum_{j\neq 1}\nu^j\log|f_{W_j}|^2)+ C$$
because of the presence of the term $\varphi_B$ in the expression of 
$\varphi_F$; hence \eqref{27} is satisfied.

\noindent $\bullet$ $\displaystyle \Theta_{h_{F}}(F)\geq 0$ by property \eqref{78}, and
for $\alpha>1/\nu ^1$ we have
\begin{equation}\label{81}\Theta_{h_{F}}(F)- {1\over \alpha}\Theta_{h_{S}}(S)\geq \Big({\alpha\nu^1- 1\over \alpha\nu^1}\Big)\Theta_{h_{F}}(F),\end{equation}
and we remark that the right hand side curvature term is greater than $0$.

\noindent $\bullet$ We have
$$\int_S|u|^2e^{-\varphi_B- (m_0-1)\varphi_{\infty}}\leq C$$
by relation \eqref{78}. Indeed, as a consequence of \eqref{78} we have 
$$|u|^2e^{-\varphi_B- (m_0-1)\varphi_{\infty}}\leq |u|^{2\over m_0}e^{-\varphi_B}$$
so the convergence of the preceding integral is due to the fact that $(S, B|_S)$ 
is klt.

\medskip
\noindent Therefore, we can apply Theorem \eqref{t-OT} and obtain an extension of $u$.
Theorem \ref{t-ext} is proved.
\end{proof}

\medskip

\begin{rem} {\rm In fact, the metric \eqref{62} of the line bundle $\cO _X(F)$ has strictly positive curvature, but the amount of positivity this metric has is ${1\over km_0}\omega$, and the estimates for the extension we obtain under these circumstances are not useful,
in the sense that the constant $C(\delta)$ in \eqref{66} becomes something like $Ck^2$.
}
\end{rem}

\section{Further consequences, I}
In this section we derive a few results which are related to Theorem \ref{t-ext}.
Up to a few details (which we will try to highlight), their proof is similar to that of \eqref{t-ext}, so our presentation will be brief.

We first remark that the arguments in Sections 4 and 5 have the following consequence.

\begin{theorem}Let $\{S, Y_j\}$ be a finite set of hypersurfaces having normal crossings. Let $B=\sum b^jY_j$ where $0< b^j< 1$ is a set of rational numbers, such that:

\begin{enumerate}
\item [{\rm (i)}] The bundle $K_X+S+B$ is pseudo-effective, and $\displaystyle S\not\in
N_\sigma(K_X+S+B)$.

\smallskip

\item [{\rm (ii)}] We have $K_X+S+B\equiv \sum_j\nu^jW_j$, where $\nu^j> 0$ and
$S\subset  {\rm Supp}(\sum \nu ^jW_j)\subset {\rm Supp}(S+B)$.

\end{enumerate}

\smallskip

\noindent Then $K_X+S+B$ admits a metric $h= e^{-\phi}$ with positive curvature and well-defined restriction to $S$.

\end{theorem}

\noindent As one can see, the only modification we have to operate for the proof of \eqref{t-ext} is to replace the family of sections
$u^{\otimes k}\otimes s_A$ with a family of sections approximating a closed positive current on $S$, whose existence is insured by the hypothesis (i).

\medskip

\noindent The next statement of this section is an $\bR$-version of \eqref{t-ext}.

\begin{theorem}\label{t-6}Let $\{S, Y_j\}$ be a finite set of hypersurfaces having normal crossings. Let $0< b^j< 1$ be a set of real numbers. Consider the $\bR$-divisor $B:= \sum_jb^jY_j$, and assume that the following properties are satisfied.

\begin{enumerate}

\item [{\rm (a)}] The $\bR$-bundle $K_X+S+B$ is pseudo-effective, and 
$\displaystyle S\not\in
N_\sigma(K_X+S+B)$.

\item [{\rm (b)}] There exists an effective $\bR$-divisor $ \sum_j \nu^jW_j$, numerically equivalent with $K_X+S+B$, such that $S\subset \{W_j\}\subset {\rm Supp} (S+B)$.

\item [{\rm (c)}] The bundle $K_S+ B|_{S}$ is $\bR$-linearly equivalent to an effective divisor
say $D:= \sum_j\mu^jZ_j$,  such that $\pi^\star(D)+ \wt E|_{\wt S}\geq \Xi$
(we use here the notations and conventions of \eqref{t-ext}).
\end{enumerate}

\noindent Then $K_X+S+B$ is $\bR$-linearly equivalent to an effective divisor whose support do not contain $S$.
\end{theorem}

\medskip

\begin{proof} By a completely standard Diophantine
approximation argument, we deduce the following fact. For any $\eta> 0$,
there exists rational numbers $b^j_\eta$, $\nu^j_\eta$ and $\mu^j_\eta$
such that the following relations are satisfied.
\begin{itemize}
\item  We have $K_X+S+ B_\eta\equiv \sum_j\nu^j_\eta W_j$, where
$B_\eta:= \sum_jb^j_\eta Y_j$;
\item  The bundle $K_S+ B_\eta|_S  $ is $\mathbb Q$-linearly equivalent to
$\sum_j\mu^j_\eta Z_j$;
\item  Let $q_\eta$ be the common denominator of $b^j_\eta$, $\nu^j_\eta$ and $\mu^j_\eta$; then we have
\begin{equation}\label{81b}q_\eta\Vert b_\eta-b\Vert< \eta, \quad q_\eta\Vert \nu_\eta-\nu\Vert< \eta, \quad
q_\eta\Vert \mu_\eta-\mu\Vert< \eta.\end{equation}\end{itemize}

\smallskip

\noindent Let $u_\eta\in H^0\big(S, \mathcal O _S(q_\eta(K_S+ B_\eta|_S))\big)$ be the section associated to the
divisor $\sum_j\mu^j_\eta Z_j$.
We invoke again the extension theorems in \cite{HM10}
(see also \cite[1.H, 1.G]{Paun08}): as a consequence, the section
$$u_\eta^{\otimes k_\eta}\otimes s_{A, i}^{q_\eta}$$
of the bundle $k_\eta q_\eta(K_S+ B_\eta)+ q_\eta A$
extends to $X$, where $k_\eta$ is a sequence of integers such that $k_\eta\to \infty$ as $\eta\to 0$.

We use the corresponding extensions $\displaystyle \{ U^{(k_\eta ,q_\eta)}_i\}_{i=1,\ldots ,M_\eta}$
in order to define a metric $h_\eta$ on
$$\displaystyle K_X+S+B_\eta+{1\over k_\eta}A,$$ with semi-positive curvature current, and whose restriction to $S$ is equivalent with
$\displaystyle \log|u_\eta|^{2\over q_\eta}$.

\noindent The proof of Theorem \ref{t-5} shows that for each 
$\eta> 0$, there exists a function $f_\eta\in L^1(X)$ such that
\smallskip
\begin{enumerate}

\item [{$(1_\eta)$}] We have $\max_Xf_\eta = 0$, as well as $\displaystyle \Theta_h(K_X+S+B)+ \sqrt{-1}\ddbar f_\eta\geq -{2\over k_\eta}\omega$.
\smallskip

\item [{$(2_\eta)$}] The restriction $f_\eta|_S$ is well-defined, and we have
\begin{equation}\label{82}f_{\eta}|_S\geq C+ \log|u_\eta|^{2\over q_\eta}.\end{equation}

\end{enumerate}

\smallskip
\noindent Passing to a subsequence we may assume that there is a limit of $f_\eta$; hence we infer the existence of a function
$f_\infty$, such that $\Theta_h(K_X+S+B)+ \sqrt{-1}\ddbar f_\infty\geq 0$,
and such that
\begin{equation}\label{83}f_{\infty}|_S\geq C+ \log\Big(\prod_j|f_{Z_j}|^{2\mu^j}\Big) .
\end{equation}

\noindent We will next use the metric $h_\infty := e^{-f_\infty}h$ in order to extend the section
$u_\eta$ above, as soon as $\eta$ is small enough. We write
\begin{equation}\label{84}q_\eta(K_X+S+B_\eta)= K_X+S+B_\eta+ (q_\eta-1)(K_X+S+B)+ (q_\eta-1)(B_\eta-B).\end{equation}
Consider a metric on $F_\eta:= B_\eta+ (q_\eta-1)(K_X+S+B)+ (q_\eta-1)(B_\eta-B)$ given by the following expression
\begin{equation}\label{85}\sum_j\big(q_\eta b^j_\eta- (q_\eta- 1)b^j\big)\log |f_{Y_j}|^2+ (q_\eta-1)\varphi_{f_\infty}.\end{equation}
In the expression above, we denote by $\displaystyle \varphi_{f_\infty}$
the local weight of the metric $h_\infty$. By the maximum procedure used e.g.
at the end of the proof of \eqref{t-5}, we can assume that
\begin{equation}\label{86}\varphi_{f_\infty}\geq \log\Big(\prod_j|f_{W_j}|^{2\nu^j}\Big).\end{equation}
The metric in \eqref{85}
has positive curvature, and one can easily check that the other curvature hypothesis are also verified (here we assume that $\eta\ll 1$, to insure
the positivity of $q_\eta b^j_\eta- (q_\eta- 1)b^j$ for each $j$). The integrability requirement \eqref{28} is satisfied, since we have
\begin{equation}\label{87}\int_S{\prod_j|f_{Z_j}|^{2q_\eta\mu^j_\eta}\over
\prod_j|f_{Y_j}|^{2q_\eta b^j_\eta- 2(q_\eta- 1) b^j}\prod_j|f_{Z_j}|^{2(q_\eta-1)\mu^j}}
<\infty\end{equation}
for all $\eta\ll 1$, by the Dirichlet conditions at the beginning of the proof. 
Hence each section
$u_\eta$ extends to $X$, and the proof of \eqref{t-6} is finished by the usual convexity argument.
\end{proof}

\medskip
\noindent Our last statement concerns a version of \eqref{t-ext} whose hypothesis are
more analytic; the proof is obtained {\sl mutatis mutandis}.
\smallskip

\begin{theorem}\label{t-6.3}Let $\{ S, Y_j\}$ be a finite set of hypersurfaces having normal crossings. Let $0< b^j< 1$ be a set of rational numbers. Consider the $\bQ$-divisor $B:= \sum_jb^jY_j$, and assume that the following properties are satisfied.

\begin{enumerate}

\item [{\rm (a)}] The bundle $K_X+S+B$ is pseudo-effective, and 
$S\not\in N_\sigma(K_X+S+B)$.
\smallskip

\item [{\rm (b)}] There exists a closed positive current $T\in \{K_X+S+B\}$ such that
\begin{enumerate}
\smallskip

\item [{\rm (b.1)}] We have $T= \nu^1[S]+ \Lambda_T$, with $\nu_1> 0$ and
$\Lambda_T$ is positive.
\smallskip

\item [{\rm (b.2)}] The following inequality holds
\begin{equation}\label{88}\varepsilon_0\varphi_{\Lambda_T}\geq \varphi_B- C
\end{equation}

where $\varepsilon_0$ and $C$ are positive real numbers.
\end{enumerate}
\smallskip

\item [{\rm (c)}] The bundle $m_0(K_S+ B)$ has a section $u$, whose zero divisor
$D$ satisfies the relation $\pi^\star(D)+ \wt E|_{\wt S}\geq \Xi$
(we use here the notations and conventions in \eqref{t-ext}).

\end{enumerate}
\smallskip
\noindent Then the section $u$ extends to $X$; in particular, the bundle
$K_X+S+B$ is $\bQ$-effective, and moreover it has a section non-vanishing identically on $S$.

\end{theorem}

\medskip

\noindent We remark here that in the proof of statement \eqref{t-6.3} we are using the full force of Theorem \ref{t-OT}. The hypothesis  above corresponds to the fact that $\{ W_j\}\subset
\{ S, Y_j\}$ in \eqref{t-ext}.

\section{Proof of Theorem \ref{t-good1}}\label{s-good}
We begin by proving \eqref{c-ext}:
\begin{proof}[Proof of \eqref{c-ext}] Let $f:X'\to X$ be a log resolution of $(X,S+B)$ and write $K_{X'}+S'+B'=f^*(K_X+S+B)+E$ where $S'$ is the strict transform of $S$, $B'$ and $E$ are effective $\mathbb Q$-divisors with no common components. 
Then $(X',S'+B'+\epsilon E)$ is also a log smooth plt pair for some rational number $0<\epsilon \ll 1$. We may also assume that the components of $B'$ are disjoint.

Since $K_X+S+B$ is pseudo-effective, so is  $K_{X'}+S'+B'+\epsilon E$.
Since $K_X+S+B$ is nef, $N_\sigma(K_X+S+B)=0$ and so $$N_\sigma(K_{X'}+S'+B'+\epsilon E)=N_\sigma(f^*(K_X+S+B)+(1+\epsilon )E)=(1+\epsilon )E.$$ 
In particular $S'$ is not contained in $N_\sigma(K_{X'}+S'+B'+\epsilon E)$ and $N_{\sigma }(|| K_{X'}+S'+B'+\epsilon E||_{S'})= (1+\epsilon )E|_{S'} $ so that
$$\Xi =(B'+\epsilon E)|_{S'}\wedge (1+\epsilon )E|_{S'}=\epsilon E|_{S'}.$$ Since there is an effective $\mathbb Q$-divisor $D\sim _{\mathbb Q} K_X+S+B$ such that $S\subset {\rm Supp}(D)\subset {\rm Supp}(S+B)$, then $D'=f^*D+(1+\epsilon )E\sim _{\mathbb Q} K_{X'}+S'+B'+\epsilon E$ is an effective $\mathbb Q$-divisor such that $S'\subset {\rm Supp} (D')\subset {\rm Supp} (S'+B'+\epsilon E)$.
By \eqref{t-ext}, $$|m(K_{X'}+S'+B'+\epsilon E)|_{S'}\supset 
|m(K_{S'}+ B'|_{S'}+E|_{S'})|+m\epsilon E$$ for any $m>0$ sufficiently divisible.
Let $\sigma \in H^0(S, \mathcal O _S(m(K_S+B_S)))$ and $\sigma '\in H^0(S',\mathcal O _{S'}(m(K_{S'}+B'|_{S'}+(1+\epsilon) E|_{S'})))$ be the corresponding section. By what we have seen above, this section lifts to a section $\tilde \sigma ' \in H^0(X',\mathcal O _{X'}(m(K_{X'}+S'+B'+(1+\epsilon ) E)))$.
Let $\tilde \sigma =f_*\tilde \sigma ' \in H^0(X,\mathcal O _{X}(m(K_{X}+S+B)))$, then $\tilde \sigma |_S=\sigma$.
\end{proof} 
Theorem \ref{t-good1} is an immediate consequence of the following:
\begin{theorem}\label{t-good}Assume \eqref{c-dlt}$_n$ and assume \eqref{c-abb}$_{n-1}$ for semi-dlt log pairs. 
Let $(X,\Delta )$ be an $n$-dimensional klt pair such that $\kappa (K_X+\Delta ) \geq 0$ then $(X,\Delta)$ has a good minimal model.\end{theorem}

\begin{proof}
We proceed by induction on the dimension. In particular we may also assume that \eqref{c-gmm}$_{n-1}$ holds.

If $\kappa (K_X+\Delta )=\dim X$, then $(X,\Delta)$ has a good minimal model by \cite{BCHM10}.

If $0<\kappa (K_X+\Delta )<\dim X$, then $(X,\Delta )$ has a good minimal model by \cite{Lai10}.
 
We may therefore assume that $\kappa (K_X+\Delta )=0$. We write $K_X+\Delta \sim _{\mathbb Q}D\geq 0$. Passing to a resolution, we may assume that $(X,\Delta+D)$ is log smooth. We will need the following.
\begin{lemma}\label{l-red} If $D=\sum_{i\in I} d_iD_i$, then it suffices to show that $(X,\Delta ')$ has a  good minimal model where $\Delta'$  is a $\mathbb Q$-divisor of the form $\Delta '=\Delta +\sum g _iD_i$ 
such that $g_i\geq 0$ are positive rational numbers and either \begin{enumerate}
\item $(X,\Delta ')$ is klt, or 
\item $(X,\Delta ')$ is dlt and $g_i>0$ for all $i\in I$.\end{enumerate}
\end{lemma}
\begin{proof} If $g_i>0$ for all $i\in I$, then for any rational number $0<\epsilon \ll 1$, we have
$(1-\epsilon)(K_X+\Delta ')\sim _{\mathbb Q}K_X+\Delta '-\epsilon (D+\sum g _iD_i)$ where 
$$\Delta \leq  \Delta '':=\Delta '-\epsilon (D+\sum g _iD_i)\leq \Delta '$$ and $(X,\Delta '')$ is klt. Since $K_X+\Delta ''\sim _{\mathbb Q}(1-\epsilon)(K_X+\Delta ')$, then $(X,\Delta ')$ has a good minimal model if and only if $(X,\Delta '')$ has a good minimal model. 
Replacing $\Delta '$ by $\Delta ''$, we may therefore assume that $(X,\Delta ')$ is klt.

Note that$$K_X+\Delta \leq K_X+\Delta '\leq K_X+\Delta +gD\sim _{\mathbb Q}(1+g)(K_X+\Delta )$$ for some rational number $g>0$. Therefore, $\kappa (K_X+\Delta ')=0$. In particular, by \eqref{e-ex}, $${\bf Fix}(K_X+\Delta ')={\rm Supp }(D+\Delta '-\Delta )={\rm Supp }(D).$$ Suppose that $(X,\Delta ')$ has a good minimal model $\phi :X\dasharrow X'$. Passing to a resolution, we may assume that $\phi$ is a morphism and that $${\rm Supp }(D)={\bf Fix}(K_X+\Delta ')
=
{\rm Exc}(\phi )$$ where ${\bf Fix} $  denotes the support of the divisors contained in the stable base locus. 

We now run a $K_X+\Delta$-minimal model program with scaling over $X'$. By \cite{BCHM10} (cf. \eqref{t-term} and \eqref{r-rel}), this minimal model program terminates. Therefore, we may assume that
${\bf Fix}(K_X+\Delta /X')=0$.
Since $D$ is exceptional over $X'$, if $D\ne 0$, then by \cite[3.6.2.1]{BCHM10}, there is  a component $F$ of $D$ which is covered by curves $\Sigma$ such that $D\cdot \Sigma <0$. This implies that ${\bf Fix}(K_X+\Delta /X')={\bf Fix}(D/X')$ is non-empty; a contradiction as above. Therefore, $D=0$ and hence $K_X+\Delta \sim _{\mathbb Q}0$.
\end{proof}

We let $S=\sum S_i$ be the support of $D$ and we let $S+B=S\vee \Delta$ (i.e. ${\rm mult}_P(S+B)={\rm max}\{  {\rm mult}_P (S), {\rm mult}_P (\Delta ) \}$) and $G=D+S+B-\Delta$ so that $K_X+S+B\sim _{\mathbb Q}G\geq 0$ and ${\rm Supp}(G)={\rm Supp }(S)$. By \eqref{l-red}, it suffices to show that $(X,S+B)$ has a good minimal model.
We now run a minimal model program with scaling of a sufficiently ample 
divisor. By \eqref{c-termination}$_{n-1}$ and \eqref{t-st}$_n$, this minimal model terminates
 giving a birational contraction $\phi:X\dasharrow X'$ such $(X',S'+B':=\phi _*(S+B))$ is dlt and $K_{X'}+S'+B'$ is nef.

If $S'=0$, then $K_{X'}+S'+B'\sim _{\mathbb Q}0$ and we are done by \eqref{l-red}.
Therefore, we may assume that $S'\ne 0$. Note that if we let $K_{S'}+B'_{S'}:=(K_{X'}+S'+B')|_{S'}$ then the pair $({S'},B'_{S'})$ is semi-dlt.
By \eqref{c-abb}$_{n-1}$, we have that $H^0(S',\mathcal O _{S'}(m(K_{S'}+B'_{S'})))\ne 0$ for all sufficiently divisible integers $m>0$.
By \eqref{c-dlt}$_n$, the sections of $H^0(S',\mathcal O _{S'}(m(K_{S'}+B'_{S'})))$ extend to $H^0(X',\mathcal O _{X'}(m(K_{X'}+S'+B')))=H^0(X',\mathcal O _{X'}(mG'))$ and hence $S'\not \subset {\rm Supp}(G')$ contradicting the fact that $\kappa (G')=0$.

\end{proof}
\section{Further remarks}
The goal of this section is to show that assuming the Global ACC Conjecture (cf. \eqref{c-acc1} below), one can reduce \eqref{c-abb} to the following weaker conjecture:
\begin{conjecture}\label{c-abbw} Let $X$ be a smooth projective variety. If $K_X$ is pseudo-effective then 
$\kappa (K_X)\geq 0$.
\end{conjecture}
We will need the following:
\begin{conjecture}[Global ACC]\label{c-acc1} Let  $d\in \mathbb N$ and  $I \subset [0,1]$ be a set satisfying the ACC. Then there is a finite subset $I_0\subset I$ such that if \begin{enumerate}
\item $X$ is a projective variety of dimension $d$,
\item $(X,\Delta  )$ is log canonical,
\item $\Delta =\sum \delta _i \Delta _i $ where $\delta _i \in I$,
\item $K_X+\Delta \equiv 0$,\end{enumerate}
then $\delta _i \in I_0$.
\end{conjecture}
\begin{rem} Recall that a set satisfies the ACC (i.e. the ascending chain condition) if any non decreasing sequence is eventually constant.
A proof of \eqref{c-acc1} has been announced by Hacon, M\textsuperscript{c}Kernan and Xu. They also show that \eqref{c-acc1} implies the ACC for log canonical thresholds cf. \eqref{c-acc} below.\end{rem} 
 \begin{conjecture}[ACC for LCTs]\label{c-acc} Let $d\in \mathbb N$, $\Gamma \subset [0,1]$ be a set satisfying the DCC and $S\subset \mathbb R _{\geq 0}$ be a finite set. Then the set$$\{{\rm lct}(D,X,\Delta )|\ (X,\Delta )\ {\rm is\ lc},\ \dim X=d,\ \Delta \in \Gamma,\ D\in S\}$$satisfies the ACC. Here $D$ is $\mathbb R$-Cartier and $\Delta\in \Gamma$ (resp. $\ D\in S$) means $\Delta=\sum \delta _i\Delta _i$ where $\delta _i\in \Gamma$ (resp. $D=\sum d_iD_i$ where $d_i\in S$)  and ${\rm lct}(D,X,\Delta )={\rm sup}\{ t\geq 0|(X,\Delta +tD)\ {\rm is\ lc}\}$. \end{conjecture}

\begin{rem}
Following \cite{Birkar07} it seems likely that \eqref{c-abb}  in dimension $n$  and \eqref{c-acc} in dimension $n-1$ imply the termination of flips for any pseudo-effective $n$-dimensional lc pair.\end{rem}
\begin{defn}Let $(X,\Delta )$ be a projective klt pair and $G$ an effective $
\Q$-Cartier divisor such that $K_X+\Delta +tG$ is pseudo-effective for some $
t\gg 0$.

Then the pseudo-effective threshold $\tau =\tau (X,\Delta;G)$ is given by
$$\tau ={\rm inf}\{ t\geq 0 | K_X+\Delta +tG\ {\rm is \ pseudo-effective}\}.$$
\end{defn}

\begin{proposition}\label{p-threshold} Assume \eqref{c-acc1}.
If  $ \tau=\tau (X,\Delta;G)$ is the pseudo-effective threshold positive, then $\tau$ is rational.\end{proposition}
\begin{proof} We may assume that $\tau =\tau (X,\Delta;G)>0$. Fix an ample divisor $A$ on $X$ 
and for any $0\leq x \leq \tau$ let $y=y(x)= \tau (X,\Delta+xG;A)$.
Then $y(x)$ is a continuous function such that $y(\tau )=0$ and $y(x)\in \mathbb Q$ for all  rational numbers $0\leq x <\tau$; moreover, for all $0\leq x <\tau$, the $K_X+\Delta+xG$ minimal model program 
with scaling ends with a $K_X+\Delta +xG+yA$-trivial Mori fiber space $g\circ
 f:X\dasharrow Y_x\to Z_x$ (cf. \cite{BCHM10} or \eqref{t-term}).
Let $F=F_x$ be the general fiber of $g$, then $$K_F+\Delta _F+x G_F+yA_F:=(K_{Y_x}+f_
*(\Delta+xG+yA))|_F\equiv 0$$ and $K_F+\Delta _F+\tau G_F$ is pseudo-effective.
Therefore $K_F+\Delta _F+\eta G_F\equiv 0$ for some $x<\eta =\eta(x)\leq \tau
$. By \eqref{c-acc1}, we may assume that $\eta =\eta (x)$ is constant and hence $
\eta =\tau$. In particular $\tau \in \Q$. 
\end{proof}

\begin{theorem}\label{t-abbw} Assume \eqref{c-gmm}$_{n-1}$, \eqref{c-abbw}$_n$ and \eqref{c-acc1}$_n$. 
Let $(X,\Delta)$ be an $n$-dimensional klt pair such that $K_X+\Delta$ is pseudo-effective. Then $\kappa(K_X+\Delta)\geq 0$.
\end{theorem}
\begin{proof}
We may also assume that $K_X$ is not pseudo-effective and $\Delta \ne 0$. 
Replacing $X$ by a birational model, we may assume that $(X,\Delta )$ is log 
smooth.
Let $\tau = \tau (X,0;\Delta )$. Note that $\tau >0$.
By \eqref{p-threshold} and its proof, there is a birational contraction $f:X\dasharrow Y$, a rational number $\tau >0$  and a  $K_X+\tau \Delta$-trivial Mori fiber space $Y\to Z$. (Here, for $0<\tau -x \ll 1$, we have denoted $f_x$ by $f$, $Y_x$ by $Y$ and $Z_x$ by $Z$.)

Assume that $\dim X>\dim Z>0$. After possibly replacing $X$ by a resolution, we may assume that $f:X\to Y$ is a morphism.
Let $C$ be a general complete intersection curve on $F$ the general fiber of $Y\to Z$. We may assume that $C\cap f({\rm Exc}(f))=\emptyset$ and so we have an isomorphism $C'=f^{-1}(C)\to C$. Thus $$(K_X+\tau \Delta )\cdot C'=(K_Y+\tau f_* \Delta )\cdot C=0.$$
In particular $K_X+\tau \Delta$ is not big over $Z$.
Since $\dim (X/Z)<\dim X$, by \eqref{c-gmm}$_{n-1}$, we may assume that the general fiber of $(X,\tau \Delta)$ has a good minimal model over $Z$.
By \cite{Lai10}, we have a good minimal model $h:X\dasharrow W$ for $(X,\Delta )$ over $Z$. Let $$r:W\to V:={\rm Proj}_Z R(K_W+\tau h_* \Delta )$$ be the corresponding morphism over $Z$.
By \cite[0.2]{Ambro05}, we may write $K_W+\tau h_*\Delta=r^*(K_V+B_V)$ where $(V,B_V)$ 
is klt. By induction on the dimension $\kappa (K_V+B_V)\geq 0$. Therefore $$\kappa (K_X+\Delta )\geq \kappa (K_X+\tau \Delta )=\kappa (K_V+B_V)\geq 0.$$

Therefore, we may assume that $\dim Z=0$ (for all $0<\tau -x \ll 1$). We claim that we may assume that $f$ is $K_X+\tau \Delta$-non-positive. Grant this for the time being, then by the Negativity Lemma (since $K_Y+\tau f_*\Delta$ is nef) it is easy to see that $\kappa ( K_X+\tau \Delta)=\kappa (K_Y+\tau f_*\Delta)$ (cf. \eqref{l-eq}). But as $\rho (Y)=1$, we have that $\kappa (K_Y+\tau f_*\Delta)\geq 0$ as required.
To see the claim, we first of all notice that for some fixed $0<y=y(x) \ll 1$ and any $0<\tau-x'\ll \tau -x$, we have 
$${\rm Supp} (N_\sigma (K_X+\tau \Delta))= {\rm Supp} (N_\sigma (K_X+ \tau \Delta+yA))$$ $$\qquad \supset {\rm Supp} (N_\sigma (K_X+ x'\Delta+y'A))$$ (since $(\tau -x')\Delta +(y-y')A$ is ample).
It follows that we may assume that ${\rm Supp}(N_\sigma (K_X+x_i\Delta +y_iA))$ is fixed for some sequence $0<x_i<\tau$ where $y_i=y(x_i)$ and $\lim x_i=\tau$. Thus, we may assume that all $Y_i:=Y_{x_i}$ are isomorphic in codimension $1$.
As observed above, we may also assume that $\dim Z_i=0$ and hence $K_{Y_i}+{f_i}_*(x_i\Delta +y_iA)\equiv 0$.
Let $E$ be any $f$-exceptional divisor for $f:X\dasharrow Y:=Y_1$. We must show that $f$ is $K_X+\tau \Delta$-non-positive i.e. that $a(E,Y,\tau f_* \Delta)\geq a(E,X,\tau \Delta)$.
For any $i>0$, since $K_{Y_{i}}+{f_{i}}_*(x_i\Delta +y_iA)\equiv 0$, we have $$a(E,Y,  {f}_*(x_i\Delta +y_iA))= a(E,Y_{i},  {f_{i}}_*(x_i\Delta +y_iA))\leq a(E,X,x_i\Delta +y_iA).$$
Passing to the limit, we obtain the required inequality.

\end{proof}
Finally we recall the following important application of the existence of good minimal models.
\begin{theorem}\label{t-finite} Assume \eqref{c-gmm}.
Let $X$ be a smooth projective variety and $\Delta _i$ be $\mathbb Q$-divisors on $X$ such that $\sum \Delta _i$ has simple normal crossings support and
 $\lfloor \Delta _i\rfloor =0$. Let $\mathcal C\subset \{ \Delta =\sum t_i\Delta _i|0\leq t_i\leq 1\}$ be a rational polytope.

Then there are finitely many birational contractions $\phi _i:X\dasharrow X_i$ and finitely many projective morphisms $\psi _{i,j}:X_i\to Z_{i,j}$ (surjective with connected fibers) such that if $\Delta\in \mathcal C$  
and $K_X+\Delta$ is pseudo-effective, then there exists $i$ such that $\phi _i:X\dasharrow X_i$ is a good minimal model of $(X,\Delta )$ and $j$ such that $Z_{i,j}={\rm Proj}R(K_X+\Delta )$. Moreover, the closures of the sets 
$$\mathcal A_{i,j}=\{ \Delta =\sum t_i\Delta _i|K_X+\Delta \sim _{\mathbb R} \psi _{i,j}^* H_{i,j},\ {\rm with} \ H_{i,j}\ {\rm ample\ on}\ Z_{i,j}\}$$ are finite unions of rational polytopes.
\end{theorem}
\begin{proof} The proof follows easily along the lines of the proof of \cite[7.1]{BCHM10}. We include the details for the benefit of the reader.
We may work locally in a neighborhood $\mathcal C$ of any $\Delta$ as above.
Let $\phi :X\dasharrow Y$ be a good minimal model of $K_X+\Delta$  and $\psi:Y\to Z={\rm Proj}R(K_X+\Delta )$ so that $K_Y+\phi _* \Delta \sim _{\mathbb R,Z}0$. 
Since $\phi $ is $K_X+\Delta$-negative, we may assume that the same is true for any $\Delta ' \in \mathcal C$ (after possibly replacing $\mathcal C$ by a smaller subset).
We may therefore assume that for any $\Delta '\in \mathcal C$, the minimal models of $(X,\Delta )$ and $(Y,\phi _*\Delta ')$ coincide (cf. \cite[3.6.9, 3.6.10]{BCHM10}). Therefore, we may replace $X$ by $Y$ and hence assume that $K_X+\Delta$ is nef. 
Let $K_X+\Delta \sim _{\mathbb R}\psi ^* H$ where $H$ is ample on $Z$ and $\psi :X\to Z$. 
Note that there is a positive constant $\delta$ such that $H\cdot C\geq \delta$ for any curve $C$ on $Z$.
We claim that (after possibly further shrinking $\mathcal C$), for any $\Delta '\in \mathcal C$, we have 
that $K_X+\Delta '$ is nef if and only if it is nef over $Z$.

To this end, note that if $(K_X+\Delta ')\cdot C <0$ and $(K_X+\Delta )\cdot C>0$, then $(K_X+\Delta ^*)\cdot C <0$ where $\Delta ^*=\Delta +t(\Delta '-\Delta)$ for some $t>0$ belongs to the boundary of $\mathcal C$.
But then, by \eqref{t-er}, we may assume that $-(K_X+\Delta ^*)\cdot C\leq 2\dim X$. Since $(K_X+\Delta) \cdot C=H\cdot \psi _* C\geq \delta$ it follows easily 
that this can not happen for $\Delta '$ in any sufficiently small neighborhood $\Delta \subset \mathcal C'\subset \mathcal C$. We may replace $\mathcal C$ by $\mathcal C'$ and the claim follows.

Therefore, it suffices to prove the relative version of the Theorem over $Z$ cf. \cite[7.1]{BCHM10}. By induction on the dimension of $\mathcal C$,  we may assume the theorem holds (over $Z$) for the boundary of $\mathcal C$. 
For any $\Delta \ne \Delta '\in \mathcal C$ we can choose $\Theta $ on the boundary of $\mathcal C$ such that $$\Theta -\Delta=\lambda (\Delta '-\Delta),\qquad 0<\lambda .$$
Since $K_X+\Delta \sim _{\mathbb R, Z}0$, we have $$K_X+\Theta \sim _{\mathbb R, Z}\lambda(K_X+\Delta ').$$ Therefore $K_X+\Theta$ is pseudo-effective over $Z$ if and only if $K_X+\Delta'$ is pseudo-effective over $Z$ and the minimal models over $Z$ of $K_X+\Theta$ and $K_X+\Delta'$ coincide (cf. \cite[3.6.9, 3.6.10]{BCHM10}). It is also easy to see that if $\psi ':X\to Z'$ is a morphism over $Z$, then  $K_X+\Theta\sim _{\mathbb R}{\psi '}^*H'$ for some ample divisor $H'$ on $Z'$ if and only if $K_X+\Delta'\sim _{\mathbb R}{\psi '}^*(\lambda H')$.
The theorem now follows easily. 
\end{proof}
\begin{theorem}\label{t-cox} Assume \eqref{c-gmm}.
Let $X$ be a smooth projective variety and $\Delta _i$ be $\mathbb Q$-divisors on $X$ such that $\sum \Delta _i$ has simple normal crossings support and
 $\lfloor \Delta _i\rfloor =0$. Then the adjoint ring
$$R(X;K_X+\Delta _1,\ldots , K_X+\Delta _r)$$ is finitely generated.
\end{theorem}
\begin{proof} 
This is an easy consequence of \eqref{t-finite}, see for example the proof of \cite[1.1.9]{BCHM10}.
\end{proof}
\begin{corollary} With the notation of \eqref{t-finite}. Let $P$ be any prime divisor on $X$ and $\mathcal C ^+$ the intersection of $\mathcal C$ with the pseudo-effective cone. Then the function $\sigma _P:\mathcal C ^+\to \mathbb R_{\geq 0}$ is continuous and piecewise rational affine linear.\end{corollary}
\begin{proof} Immediate from \eqref{t-finite}.\end{proof}


\begin{thebibliography}{Kawamata85}
\bibitem[Ambro04]{Ambro04}
F. Ambro, {\it Nef dimension of minimal models.}
Math. Ann. 330 (2004), no. 2, 309–322. 
\bibitem[Ambro05]{Ambro05}
F. Ambro, {\it The moduli $b$b-divisor of an lc-trivial fibration.}  Compos. Math.  141  (2005),  no. 2, 385–403. 


\bibitem[Birkar07]{Birkar07}
C. Birkar, {\it Ascending chain condition for log canonical thresholds and termination of log flips.} Duke Math. J.  136  (2007),  no. 1, 173--180.
\bibitem[Birkar09]{Birkar09}
C. Birkar, {\it On existence of log minimal models II.} arXiv:0907.4170
\bibitem[Berndtsson96]{BoB96}B. Berndtsson, \emph{On the Ohsawa-Takegoshi extension theorem} Ann.\ Inst.\ Fourier (1996).
\bibitem[BP10]{BP10}B. Berndtsson, M. P\u aun, \emph{Qualitative extensions of twisted pluricanonical forms and closed positive currents}, arXiv 2010.
\bibitem[BCHM10]{BCHM10}
C. Birkar, P. Cascini, C. Hacon, J. M$^{\rm c}$Kernan, {\it Existence of minimal models for varieties of log general type.}  J. Amer. Math. Soc.  23  (2010),  no. 2, 405--468. 
\bibitem[Claudon07]{Clod07}{B. Claudon, {\emph{Invariance for multiples of the twisted canonical bundle,}}  Ann. Inst. Fourier (Grenoble)  57,  2007.}
\bibitem[CL10]{CL10}
A. Corti, V. Lazi\'c, {\it New outlook on Mori theory, II.} ArXiv:1005.0614
\bibitem[dFH09]{dFH09}
T. de Fernex, C. Hacon, {\it Deformations of canonical pairs and Fano varieties. } arXiv:0901.0389
\bibitem[Demailly90]{Dem90}J.-P. Demailly, \emph{Singular hermitian metrics on positive line bundles,} Proc. Conf. Complex
algebraic varieties (Bayreuth, April 2�6, 1990), edited by K. Hulek, T. Peternell,
M. Schneider, F. Schreyer, Lecture Notes in Math., Vol. 1507, Springer-Verlag, Berlin,
1992.
\bibitem[Demailly97]{Dem97}J.-P. Demailly, \emph{On the Ohsawa-Takegoshi-Manivel  
extension theorem,} Proceedings of the Conference in honour of the 85th birthday of Pierre Lelong, 
Paris, September 1997.
\bibitem[Demailly06]{Dem06}J.-P. Demailly, \emph{K\"ahler manifolds and transcendental techniques in algebraic geometry,} Plenary talk and Proceedings of the Internat. Congress of Math., Madrid 2006, volume I.
\bibitem[Demailly09]{Dem09}J.-P. Demailly, \emph{Analytic methods in algebraic geometry,}
on the web page of the author, December 2009.
\bibitem[EP07]{EP07} L. Ein, M. Popa, \emph{Adjoint ideals and extension theorems,} personal communication june 2007, arXiv: 0811.4290.
\bibitem[Fujino00]{Fujino00}
O. Fujino, {\it Abundance theorem for semi log canonical threefolds.}  Duke Math. J.  102  (2000),  no. 3, 513–532.
\bibitem[Fujino07]{Fujino07}
O. Fujino, {\it Special termination and reduction to pl flips.}  Flips for 3-folds and 4-folds,  63--75, Oxford Lecture Ser. Math. Appl., 35, Oxford Univ. Press, Oxford, 2007.
\bibitem[Fukuda02]{Fukuda02} 
S. Fukuda, {\it Tsuji’s numerically trivial fibrations and abundance,} Far East
Journal of Mathematical Sciences, Volume 5, Issue 3, 2002, 247–257.
\bibitem[GL10]{GL10}
Y. Gongyo, B Lehmann, {\it Reduction maps and minimal model theory}. Preprint.
\bibitem[HK10]{HK10}
C. Hacon, S. Kovacs, {\it
Classification of higher dimensional algebraic varieties.} Oberwolfach Seminars, 41. Birkh\"auser Verlag, Basel, 2010. x+208 pp.
\bibitem[HM07]{HM07}
C. Hacon, J. M$^{\rm c}$Kernan, {\it Boundedness of pluricanonical maps of varieties of general type.}  Invent.\ Math.\ 166 (2006),  no. 1, 1–25.
\bibitem[HM10]{HM10}
C. Hacon, J. M$^{\rm c}$Kernan, {\it Existence of minimal models for varieties of log general type. II.}  J.\ Amer.\ Math.\ Soc.\ 23  (2010),  no. 2, 469--490.

\bibitem[Kawamata85a]{Kawamata85a}
Y. Kawamata, {\it Pluricanonical systems on minimal algebraic varieties.}
Invent. Math. 79 (1985), no. 3, 567–588. 
\bibitem[Kawamata85b]{Kawamata85}
Y. Kawamata, {\it The Zariski decomposition of log-canonical divisors,} Algebraic Geometry Bowdoin, Proc. Symp. Pure Math. 46 (1987), 1985, pp. 425--433.
\bibitem[Kawamata92]{Kawamata92}Y. Kawamata, \emph{Abundance theorem for minimal threefolds.}
Invent. Math. 108 (1992), no. 2, 229–246. 
\bibitem[Kawamata97]{Kawamata97}Y. Kawamata, \emph{On the cone of divisors of Calabi-Yau fiber spaces.} Internat. J. Math. 8 (1997), no. 5, 665�687. 
\bibitem[Kawamata98]{Kawamata98}Y. Kawamata, \emph{On the extension problem of pluricanonical forms.} Algebraic geometry:
Hirzebruch 70 (Warsaw, 1998), Contemp. Math., vol. 241, Amer. Math. Soc.,
Providence, RI, 1999.
\bibitem[Kawamata08]{Kawamata08}
Y. Kawamata, {\it Flops connect minimal models.}  Publ. Res. Inst. Math. Sci.  44  (2008),  no. 2, 419–423. 
\bibitem[KMM94]{KMM94}
S. Keel, K. Matsuki, J M$^{\rm c}$Kernan,
{\it Log abundance theorem for threefolds.}
Duke Math. J. 75 (1994), no. 1, 99--119.
\bibitem[Kim08]{KIM}D. Kim, \emph{$L^2$ extension of adjoint line bundle sections,} arXiv:0802.3189, to appear in Ann.\ Inst.\ Fourier.
\bibitem[Klimek91]{Klimek} M. Klimek, \emph{Pluripotential Theory}, London Mathematical Society Monographs. New Series, 6. Oxford Science Publications. The Clarendon Press, Oxford University Press, New York, 1991, xiv+266~pp.
\bibitem[KM98]{KM98}
J. Koll\'ar and S. Mori, {\it Birational geometry of algebraic varieties,} Cambridge tracts in math- ematics, vol. 134, Cambridge University Press, 1998. 
\bibitem[Koll{\'a}retal92]{Kollaretal92}
J. Koll\'ar et al., {\it Flips and abundance for algebraic threefolds,} Soci\'et\'e Math\'ematique de France, Paris, 1992, Papers from the Second Summer Seminar on Algebraic Geometry held at the University of Utah, Salt Lake City, Utah, August 1991, Ast\'erisque No. 211 (1992).
\bibitem[Lai10]{Lai10} C.-J. Lai, {\it Varieties fibered by good minimal models} arXiv:0912.3012, to appear in Math. Ann.
\bibitem[Lelong69]{Lelong69} P. Lelong, {\it Plurisubharmonic Functions and Positive 
Differential Forms}, Gordon and Breach, Dunod, New York, Paris, 1969.
\bibitem[Lelong71]{Lelong71} P. Lelong, {\it
\'El\'ements extr\'emaux dans le c\^one des courants positifs ferm\'es de type $(1,\,1)$ et fonctions plurisousharmoniques extr\'emales},
C.~R.\ Acad.\ Sci.\ Paris S\'er. A--B 273 (1971), A665--A667. 
\bibitem[Manivel93]{Man93}
L. Manivel, \emph{Un th\'eor\`eme de prolongement $L^2$ de sections holomorphes d'un 
 fibr\'e hermitien} Math.\ Zeitschrift, 1993.
 \bibitem[MV07]{MV07} J. McNeal, D. Varolin; \emph{Analytic inversion of adjunction: $L\sp 2$ extension theorems with gain}, Ann.\ Inst.\ Fourier (Grenoble)  {\bf 57}  (2007),  no. 3, 703--718.
\bibitem[Miyaoka88]{Miyaoka88}
Y. Miyaoka, {\it Abundance conjecture for minimal threefolds: $\nu = 1$ case.} Comp. Math. 68, 203–220 (1988)
\bibitem[Nakayama04]{Nakayama04}
N. Nakayama, {\it Zariski-decomposition and abundance,} MSJ Memoirs, vol. 14, Mathematical Society of Japan, Tokyo, 2004.
\bibitem[OT87]{OT87} T. Ohsawa, K. Takegoshi, \emph{On the extension of $L^2$
holomorphic functions} Math.\ Z., 1987.
\bibitem[Ohsawa03]{Oh03} T. Ohsawa, \emph{On the extension of $L\sp 2$ holomorphic functions. VI. A limiting case, }
Contemp.\ Math.,  Amer. Math. Soc., Providence, 2003.
\bibitem[Ohsawa04]{Oh04} T. Ohsawa, \emph{Generalization of a precise $L^2$ division theorem,} Complex analysis in several
variables, Memorial Conference of Kiyoshi Oka's Centennial Birthday, 249--261, Adv.
Stud. Pure Math., 42, Math. Soc. Japan, Tokyo, 2004.
\bibitem[Paun07]{Paun07}
M. P\u aun, {\it Siu's invariance of plurigenera: a one-tower proof,} J. Diff.\ Geom.\, 76 
(2007), 485--493. 
\bibitem[Paun08]{Paun08}
M. P\u aun, {\it Relative critical exponents, non-vanishing and metrics with minimal singularities.} arXiv:0807.3109.
\bibitem[Shokurov09]{Shokurov09}
V.V. Shokurov, {\it Letters of a bi-rationalist. VII. Ordered termination.} (Russian)  Tr. Mat. Inst. Steklova  264  (2009),  Mnogomernaya Algebraicheskaya Geometriya, 184--208.
\bibitem[Siu74]{Siu74} Y.T. Siu, \emph{Analyticity of sets associated to Lelong numbers and the extension of closed 
positive currents,} Invent.\ Math.\, 27 (1974), 53--156. 
\bibitem[Siu76]{Siu76} Y. T. Siu, \emph{Every Stein subvariety admits a Stein neighborhood}, Invent.\ Math.\ 38 (1976) 89--100.
\bibitem[Siu98]{Siu98} Y.-T. Siu, {\it Invariance of plurigenera.} Invent. Math. 134, 661--673 (1998).
\bibitem[Siu00]{Siu00} Y.-T. Siu, {\it Extension of twisted pluricanonical sections with plurisubharmonic weight and invariance of semipositively twisted plurigenera for manifolds not necessarily of general type.} Complex geometry (G\"ottingen, 2000), 223--277, Springer, Berlin, 2002. 
\bibitem[Siu08]{Siu08} Y.-T. Siu, {\it Finite generation of canonical ring by analytic method. }
Sci. China Ser. A 51 (2008), no. 4, 481--502. 
\bibitem[Siu09]{Siu09} Y.-T. Siu, {\it Abundance conjecture} arXiv:09120576, 2010.
\bibitem[Skoda72]{Skoda}
H. Skoda, {\it Applications des techniques 
$L^2$ 
\`a la th\'eorie des id\'eaux d'une alg\`ebre de 
fonctions holomorphes avec poids}, Ann. Scient. Ec. Norm. Sup. 4e S\'erie, 5 (1972), 545--579. 
\bibitem[Takayama06]{Taka06} S. Takayama, \emph{Pluricanonical systems on algebraic varieties of general type,} Invent.\ 
Math, 165 (2006), 551--587. 
\bibitem[Takayama07]{Taka07} S. Takayama, \emph{On the invariance and lower semi-continuity of plurigenera of algebraic 
varieties,} J.\ Algebraic Geom, 16 (2007), 1�18. 
\bibitem[Tian87]{Tian}G. Tian, \emph{
On K\"ahler-Einstein metrics on certain K\"ahler manifolds with $c_1 (M ) > 0$}, 
Invent. Math, 89 (1987), 225--246. 
\bibitem[Tsuji05]{Tsuji}H. Tsuji, \emph{Extension of log pluricanonical forms from subvarieties} math.CV/0511342.
\bibitem[Var08]{Var08}D. Varolin, \emph{A Takayama-type extension theorem,} Compos. Math. 144 (2008), no. 2, 522--540. 
\end{thebibliography}
\end{document}